\documentclass[11pt,reqno,tbtags]{amsart}
\allowdisplaybreaks

\usepackage{amsthm,amssymb,amsfonts,amsmath}
\usepackage{amscd} % commutative diagrams. 
\usepackage{mathrsfs}
\usepackage[numbers, sort&compress]{natbib}
\usepackage{subfigure,graphicx}
\usepackage{vmargin, enumerate}
\usepackage{setspace}
\usepackage{paralist}
\usepackage[pagebackref=true]{hyperref}
\usepackage{color}
\usepackage{stmaryrd} % for \bbr
\usepackage[refpage,intoc]{nomencl} % for list of notation. %nopFix
\usepackage{tikz}
\usetikzlibrary{arrows}
\usetikzlibrary{decorations.markings}
\usepackage{wrapfig}
\usepackage{mathabx}
\usepackage{etoolbox}

\numberwithin{equation}{section}

\newcommand{\R}{\mathbb{R}}

\newcommand{\N}{{\mathbb N}}

\newcommand{\K}{\mathbb{K}}

\newcommand{\E}[1]{{\mathbb E}\left[#1\right]}

\newcommand{\p}[1]{{\mathbb P}\left(#1\right)}

\newcommand{\I}[1]{{\mathbf 1}_{[#1]}}
\newcommand{\set}[1]{\left( #1 \right)}
\newcommand{\Cprob}[2]{\mathbb{P}\set{\left. #1 \; \right| \; #2}}

\newcommand{\instep}{\hspace*{0.3cm}}

\newtheorem{theorem}{Theorem}[section]
\newtheorem{lem}[theorem]{Lemma}
\newtheorem{prop}[theorem]{Proposition}
\newtheorem{cor}[theorem]{Corollary}

\newtheorem{fact}[theorem]{Fact}

\newcommand\cA{\mathcal A}
\newcommand\cB{\mathcal B}

\newcommand\cG{\mathcal G}

\newcommand\cL{{\mathcal L}}
\newcommand\cM{\mathcal M}

\newcommand\cP{\mathcal P}
\newcommand\cQ{\mathcal Q}
\newcommand\cR{{\mathcal R}}
\newcommand\cS{{\mathcal S}}
\newcommand\cT{{\mathcal T}}

\newcommand{\rG}{\mathrm{G}} 
\newcommand{\rH}{\mathrm{H}}

\newcommand{\rM}{\mathrm{M}}

\newcommand{\rP}{\mathrm{P}} 
\newcommand{\rQ}{\mathrm{Q}} 
\newcommand{\rR}{\mathrm{R}} 
\newcommand{\rS}{\mathrm{S}}

\newcommand{\rV}{\mathrm{V}} 
 
\newcommand{\rX}{\mathrm{X}} 
\newcommand{\rY}{\mathrm{Y}}

\newcommand{\bM}{\mathbf{M}}

\newcommand{\bQ}{\mathbf{Q}} 
\newcommand{\bR}{\mathbf{R}} 
\newcommand{\bS}{\mathbf{S}}

\newcommand{\bn}{\mathbf{n}}

\newcommand{\refT}[1]{Theorem~\ref{#1}}
\newcommand{\refC}[1]{Corollary~\ref{#1}}

\newcommand{\refL}[1]{Lemma~\ref{#1}}

\newcommand{\refS}[1]{Section~\ref{#1}}
\newcommand{\refP}[1]{Proposition~\ref{#1}}

\newcommand{\refF}[1]{Figure~\ref{#1}}
\newcommand{\refFt}[1]{Fact~\ref{#1}}

\newcommand{\pran}[1]{\left(#1\right)}

\providecommand{\eps}{}
\renewcommand{\eps}{\varepsilon}
\providecommand{\veps}{}
\renewcommand{\veps}{\varepsilon}
\providecommand{\ora}[1]{}
\renewcommand{\ora}[1]{\overrightarrow{#1}}
\newcommand{\diam}{\mathrm{diam}}

\renewcommand{\b}{\ensuremath{\mathrm{b}}}
\renewcommand{\sb}{\ensuremath{\mathrm{sb}}}

\newcommand{\dghp}{\ensuremath{d_{\mathrm{GHP}}}}
\newcommand{\dgh}{\ensuremath{d_{\mathrm{GH}}}}

\newcommand{\dis}{\mbox{dis}}

\newcommand{\eqdist}{\ensuremath{\stackrel{\mathrm{d}}{=}}}
\newcommand{\convdist}{\ensuremath{\stackrel{\mathrm{d}}{\to}}}

\newcommand{\convp}{\ensuremath{\stackrel{\mathrm{p}}{\to}}}

\hypersetup{
    bookmarks=true,         % show bookmarks bar?
    unicode=false,          % non-Latin characters in Acrobat’s bookmarks
    pdftoolbar=true,        % show Acrobat’s toolbar?
    pdfmenubar=true,        % show Acrobat’s menu?
    pdffitwindow=true,      % page fit to window when opened
    pdftitle={My title},    % title
    pdfauthor={Author},     % author
    pdfsubject={Subject},   % subject of the document
    pdfnewwindow=true,      % links in new window
    pdfkeywords={keywords}, % list of keywords
    colorlinks=true,       % false: boxed links; true: colored links
    linkcolor=blue,          % color of internal links
    citecolor=blue,        % color of links to bibliography
    filecolor=blue,      % color of file links
    urlcolor=blue           % color of external links
}

\newcommand\urladdrx[1]{{\urladdr{\def~{{\tiny$\sim$}}#1}}}

\begingroup
  \divide\count255 by 60
  \multiply\count255 by -60
  \advance\count255 by \time
  \ifnum \count255 < 10 \xdef\oclock{\the\count1:0\the\count255}
  \else\xdef\oclock{\the\count1:\the\count255}\fi
\endgroup

\DeclareRobustCommand{\SkipTocEntry}[5]{}

\usepackage{caption}
\begin{document}

\title{Joint convergence of random quadrangulations and their cores} 
\author{Louigi Addario-Berry \and Yuting Wen}
\address{Department of Mathematics and Statistics, McGill University, 805 Sherbrooke Street West, 
		Montr\'eal, Qu\'ebec, H3A 0B9, Canada}
\email{louigi.addario@mcgill.ca\\ yutingyw@gmail.com}
\date{\today} %; revised ...
\urladdrx{http://problab.ca/louigi\\ http://www.math.mcgill.ca/ywen}

%\subjclass{60C05 (68P10,68W40)} % {Primary: <subject>; Secondary: <subject>}

\begin{abstract} 
We show that a uniform quadrangulation, its largest $2$-connected block, and its largest simple block jointly converge to the same Brownian map in distribution for the Gromov-Hausdorff-Prokhorov topology. We start by deriving a local limit theorem for the asymptotics of maximal block sizes, extending the result in \cite{BFSS}. The resulting diameter bounds for pendant submaps of random quadrangulations straightforwardly lead to Gromov-Hausdorff convergence. To extend the convergence to the Gromov-Hausdorff-Prokhorov topology, we show that exchangeable ``uniformly asymptotically negligible'' attachments of mass simply yield, in the limit, a deterministic scaling of the mass measure. 
\end{abstract}

\maketitle

%\tableofcontents

\section{Introduction}

Much work has been devoted to understanding the asymptotic properties of large random planar maps. It is conjectured, and known in several cases, that after rescaling the graph distance properly, planar maps from many families converge to the same universal metric space, the Brownian map, in the Gromov-Hausdorff-Prokhorov sense. Recently \citet{LG} and \citet{Mi} independently proved that the Brownian map is the scaling limit of several important families of planar maps, and \citet{ABA} proved that simple triangulations and simple quadrangulations also rescale to the same limit object. 

The aim of this paper is to show that random quadrangulations and their cores jointly converge to the same limit object, even after conditioning on their sizes. Before making this more precise, we state one corollary (Theorem~\ref{thm3}) of our main result: the Brownian map is again the scaling limit of random $2$-connected quadrangulations. 

Throughout the paper, all maps are embedded in the sphere $\mathbb{S}^2$ and are considered up to orientation preserving homeomorphism. A {\em rooted map} is a pair $\rM=(M,uv)$ where $M$ is a map and $uv$ is an oriented edge of $M$. A {\em quadrangulation} is a map in which every face has degree $4$. A quadrangulation is {\em $2$-connected} if the removal of any vertex does not disconnect the map. It is {\em simple} if it contains no multiple edges. Write $\cQ$, $\cR$, and $\cS$ for the set of rooted connected, $2$-connected, and simple quadrangulations, respectively. It is easy to verify that simple quadrangulations are $2$-connected, so $\cS\subset \cR\subset\cQ$. It is technically convenient to view a single edge as a $2$-connected, simple quadrangulation, and we do this.

Given a finite set $\cG$, the notation $G\in_u\cG$ means that $G$ is chosen uniformly at random from $\cG$. Given a finite rooted or unrooted map $G$ write $\mu_G$ for the uniform probability measure on the vertex set $v(G)$, and for $c > 0$, write $cG$ for the measured metric space $(v(G),c\cdot d_G,\mu_G)$, where $d_G$ denotes graph distance. Given a set $\cG$ of maps and $n \in \N$, write $\cG_n= \{G \in \cG: |v(G)|=n\}$. Finally, write $\bM = (\cM, d,\mu)$ for the measured Brownian map. (See \cite{LG} for a definition of $\bM$.)
\begin{theorem}\label{thm3}
Let $\bR_r\in_u\cR_r$, then as $r\to\infty$,
\[
\left(\frac{21}{40r}\right)^{1/4} \bR_r \convdist \bM
\]
in distribution for the Gromov-Hausdorff-Prokhorov topology.
\end{theorem}
A brief overview of the Gromov-Hausdorff-Prokhorov (GHP) distance appears in \refS{sec:ghpintro}. 

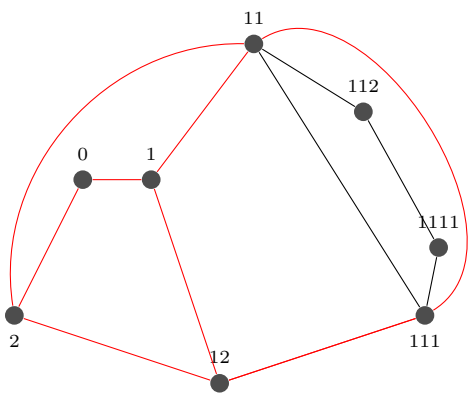
\begin{wrapfigure}[24]{L}{.45\textwidth}
\rule[5pt]{.45\textwidth}{0.5pt}
\centering
\begin{tikzpicture}[scale=0.9]
	\tikzstyle{main}=[circle,inner sep=0pt,minimum size=7pt,fill=black!70];
	
	\node[main] (0) [label={\tiny $0$}] at (-1,0) {};
	\node[main] (1) [label={\tiny $1$}] at (0,0) {};
	\node[main] (2) [ label =below:{\tiny $2$}] at (-2,-2) {};
	\node[main] (11) [label={\tiny $11$}] at (1.5,2) {};
	\node[main] (112) [label={\tiny $112$}] at (3.1,1) {};
	\node[main] (1111) [label={\tiny $1111$}] at (4.2,-1) {};
	\node[main] (111) [label=below:{\tiny $111$}] at (4,-2) {};
	\node[main] (12) [label={\tiny $12$}] at (1,-3) {};
	
	\draw (2) edge [color=red,bend left=50] (11);
    \draw  (2) edge [color=red] (0);	
    \draw  (1) edge [color=red](0);
    \draw (1) edge [color=red](11);
    \draw (11) edge (112);
    \draw (112) edge (1111);
    \draw (11) edge (111);
    \draw (1111) edge (111);
    \draw (1) edge [color=red](12);
    \draw (12) edge[color=red]  (111);	
    \draw (12) edge [color=red] (111);
    \draw (11) edge [color=red,bend left=90] (111);
    \draw (2) edge [color=red] (12);
\end{tikzpicture}
\caption{\small 
$(0, 1)$ is the root edge of $\rM$. For the total order $\prec=\prec_{\rM}$ we have, e.g., $(0,1) \prec (0,2)$, $(2,12) \prec (12,2)$  $\prec (12,111)$ $\prec (111,12)$. Also, of the two copies of edge $(11,111)$, the one succeeding $(11,2)$ in the clockwise order is smaller for $\prec$. The simple block $\rS(\rM)$, highlighted in red, has vertices $0,1,2,11,12,111$.}
\label{fig:rootrule}
\rule[8pt]{.45\textwidth}{0.5pt} 
\end{wrapfigure}

To state our main results, a little more terminology is needed. Given a rooted map $\rM=(M,uv)$, we may define a canonical total order $<_\rM$ on $v(M)$ as follows. List the vertices of $\rM$ as $u_1=u,u_2=v,\ldots,u_{|v(M)|}$ according to their order of discovery by a breadth-first search which starts from the root edge $uv$ and uses the clockwise order of edges around each vertex starting from the explored edge to determine exploration priority. (See \cite{E} for a definition of breadth-first search.) We also define a total order $\prec_{\rM}$ on the set of oriented edges of $\rM$ as follows. Let $u_iu_j\prec_{\rM} u_{i'}u_{j'}$ precisely if either (a) $u_i$ was discovered before $u_{i'}$ or (b) $i=i'$ and $u_iu_j$ has higher priority than $u_iu_{j'}$. 

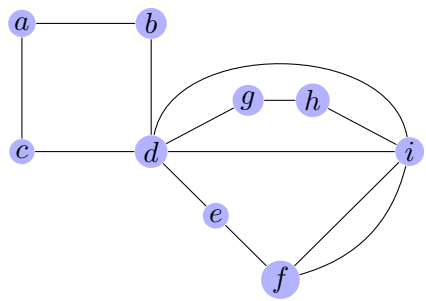
\begin{wrapfigure}[17]{R}{.44\textwidth}
\rule[5pt]{.44\textwidth}{0.5pt}
\centering
\begin{tikzpicture}[scale=0.85]
	\tikzstyle{main}=[circle,inner sep=1pt,minimum size=6pt,fill=blue!30];

	\node[main] (a) at (-2,2) {$a$} ;
	\node[main] (b) at (0,2) {$b$};
	\node[main] (c) at (-2,0) {$c$};
	\node[main] (d) at (0,0) {$d$};
	\node[main] (e) at (4,0) {$i$};
	\node[main] (g) at (1.5,0.8) {$g$};
	\node[main] (h) at (2.5,0.8) {$h$};
	\node[main] (f) at (2,-2) {$f$};
	\node[main] (i) at (1,-1){$e$};

    \draw  (a) edge (b);	
    \draw  (b) edge (d);	
    \draw  (a) edge (c);	
    \draw  (c) edge (d);	
    \draw (d) edge [bend left=80] (e);
    \draw (d) edge (g);
    \draw (g) to (h) to (e);
    \draw (d) edge (e);
    \draw (d) to (i) to (f);
    \draw (f) edge (e);
    \draw (e) edge [bend left=30] (f);	
\end{tikzpicture}
\caption{\small The $2$-connected blocks of $\rM$ are $M[\{a,b,c,d\}]^\circ$ and $M[\{d,e,f,g,h,i\}]^\circ$. The simple blocks of $M$ are $M[\{a,b,c,d,e,f,i\}]^\circ$ and $M[\{d,g,h,i\}]^\circ$.}
\label{fig:blocks}  
\rule[8pt]{.44\textwidth}{0.5pt}
\end{wrapfigure}

Fix a bipartite map $\rM=(M,uv)$.
A cycle $C$ in a map $\rM$ is {\em nearly facial} if at least one connected component of $\mathbb{S}^2\setminus C$ contains no vertices of $M$ (it may contain edges). We say $\rM$ is {\em nearly simple} if every cycle in $\rM$ with length two is nearly facial. 
Write $\rM^{\circ}=(M^{\circ},uv)$ for the map obtained by collapsing each nearly facial $2$-cycle into an edge. (This is a slight abuse of notation as the edge $uv\in e(M)$ may be collapsed with other edges in forming $M^{\circ}$, but the meaning should be clear.) 
Note that $\rM$ is nearly simple precisely if $\rM^{\circ}$ is simple -- in this case we call $\rM^{\circ}$ the {\em simple nerve} of $\rM$. 

For $A \subset v(M)$, write $M[A]$ for the submap of $M$ induced by $A$. For any edge $e\in e(M)$ with endpoints $x$ and $y$ let $B_e\subset v(M)$ be maximal subject to the constraints that $\{x,y\}\subset B_e$, and that $M[B_e]$ is $2$-connected. We call $M[B_e]^\circ$ a {\em $2$-connected block} of $\rM$. In particular, write $\rR^\bullet=\rR^{\bullet}(\rM)=(M[B_{uv}]^{\circ},uv)$ and call $\rR^\bullet$ the {\em $2$-connected root block} of $\rM$. Our choice to collapse nearly-facial 2-cycles renders this different from the standard graph theoretic definition of a 2-connected block. We make this choice as it simplifies upcoming counting arguments.

Next, for any edge $e\in e(M)$ with endpoints $x$ and $y$, consider the set $S=\{B \subset v(M): \{x,y\} \subset B, ~M[B]\mbox{ is nearly simple}\}$. Let $S' = \{B' \in S: B'\mbox{ is maximal}\}$, where maximal is with respect to the inclusion relation on $v(M)$. Then 
define $B'_e \subset v(M)$ to be the lexicographically minimal element of $S'$ with respect to the total order $<_{\rM}$. We call $M[B'_e]^{\circ}$ a {\em simple block} of $\rM$ or, more specifically, the simple block containing edge $e$. We also write $\rS^\bullet=\rS^{\bullet}(\rM)=(M[B'_{uv}]^{\circ},uv)$ and call $\rS^\bullet$ the {\em simple root block} of $\rM$. 

Write $\rR(\rM)$ (resp.~$\rS(\rM)$) for the largest $2$-connected (resp.~simple) block of $\rM$, rooted at its $\prec_{\rM}$-minimal edge, and write $\b(\rM)=|v(\rR(\rM))|$ and $\sb(\rM)=|v(\rS(\rM))|$. If there are multiple $2$-connected blocks with size $\b(\rM)$, among these blocks we take $\rR(\rM)$ to be the one whose root edge $u_iu_j$ is $\prec_{\rM}$-minimal, and use the same convention for $\rS(\rM)$. We call $\rR(\rM)$ and $\rS(\rM)$ the $2$-connected and simple cores of $\rM$, respectively.

The next theorem states that a uniform quadrangulation, its largest $2$-connected block, and its largest simple block jointly converge to the same Brownian map. (Note that the definition of $\bR_q$ in the coming theorem is different from that in Theorem~\ref{thm3}. We recycle some notation to keep the sub- and superscripts from becoming too cumbersome; we will always remind the reader when there is a possibility of ambiguity or confusion.) 

\begin{theorem}\label{thm1}
Let $\bQ_q\in_u \cQ_q$ and write $\bR_q=\rR(\bQ_q)$, $\bS_q=\rS(\bQ_q)$. Then as $q\to\infty$, 
\begin{align*}
\left(\left(\frac{9}{8q}\right)^{1/4} \bQ_q~,~\left(\frac{9}{8 q}\right)^{1/4}  \bR_q~,~\left(\frac{9}{8 q}\right)^{1/4}  \bS_q\right) \convdist \left(\bM,\bM,\bM\right)\, 
\end{align*}
in distribution for the Gromov-Hausdorff-Prokhorov topology.
\end{theorem}

The convergence of the first coordinate in \refT{thm1} was proved independently by \citet{LG} and by \citet{Mi}. The convergence of the third coordinate on its own is implied by a result by \citet{ABA}, who show that if $\bS_q$ is a uniform simple quadrangulation for all $q$, then $(3/(8|v(\bS_q)|))^{1/4} \bS_q \convdist \bM$. It is known \cite{GW,BFSS} that $|v(\bS_q)|/q \to 1/3$ in probability, so in the third coordinate the scaling factor $(9/(8q))^{1/4}$ may be replaced by $(3/(8 |v(\bS_q)|))^{1/4}$, and the convergence then follows from the result of \cite{ABA}. Similarly, the convergence of the second coordinate on its own can be deduced from \refT{thm3}. 

\refT{thm1} and \refT{thm3} both follow from a stronger ``local invariance principle'', in which the sizes of the largest $2$-connected block and largest simple block are fixed rather than random. Given integers $q \ge r \ge s \ge 1$, let 
\begin{align*}
\cQ_{q,r,s} & = \{Q \in \cQ_q: \b(Q)=r,\sb(Q)=s\}\, ,\\
\cR_{r,s} & = \{Q \in \cR_r: \sb(Q)=s\}\, . 
\end{align*}

\begin{theorem}\label{thm2}
Let $(r(q):q\in\N)$ and $(s(q):q \in\N)$ be such that $r(q) = 7q/15  + O\left(q^{2/3}\right)$ and $s(q)=q/3+O(q^{2/3})$ as $q\to\infty$. Let $\bQ_q\in_u \cQ_{q,r(q),s(q)}$ and write $\bR_q=\rR(\bQ_q)$, $\bS_q=\rS(\bQ_q)$. Then as $q\to\infty$, 
\[
\left(\left(\frac{9}{8q}\right)^{1/4} \bQ_q~,~\left(\frac{9}{8 q}\right)^{1/4}  \bR_q~,~\left(\frac{9}{8 q}\right)^{1/4}  \bS_q\right) \convdist \left(\bM,\bM,\bM\right)\, \]
in distribution for the Gromov-Hausdorff-Prokhorov topology.
\end{theorem}

We provide an outline of the proof of \refT{thm2} (our main result) in \refS{sec:sketch}. 

Now and for the remainder of the paper, fix $C > 0$ and let $(r(q):q\in\N)$ and $(s(q):q \in\N)$ be such that $|r(q)-7q/15| < C q^{2/3}$ and $|s(q)-5q/7| < C q^{2/3}$ for all $q$ sufficiently large. The scaling of $r(q)$ and $s(q)$ in \refT{thm2} is explained by the following local limit theorem for the asymptotics of maximal block sizes. 
\begin{theorem}\label{thm:blocksizes} 
Let $\bQ_q \in_u \cQ_{q}$, and write $\delta_r(q) =\frac{r(q)-7q/15}{q^{2/3}}$, $\delta_s(q)=\frac{s(r(q))-5r(q)/7}{r(q)^{2/3}}$. Then 
\[
\p{\b(\bQ_q)=r(q),\sb(\bQ_q)=s(r(q))} =  \frac {\beta\cA\left(\beta \delta_s(q)\right)}{r(q)^{2/3}}\frac {\beta' \cA\left(\beta' \delta_r(q)\right)}{q^{2/3}}(1+o(1))~,
\]
where $\beta$ and $\beta'$ are positive constants given in Propositions~\ref{schema2} and~\ref{schema1} respectively, $\cA: \R \to [0,1]$ is a density. 
\end{theorem}
Here $o(1)$ denotes a function tending to zero whose decay may depend on $C$, but we omit this dependence from the notation.
We prove \refT{thm:blocksizes} using the machinery developed by \citet{BFSS}, based on singularity analysis of generating functions, in \refS{sec:schema}. 
\refT{thm1} follows from \refT{thm2}, \refT{thm:blocksizes}, and an easy averaging argument. We similarly deduce \refT{thm3} by averaging over the second coordinate in the next proposition. 
\begin{prop}\label{mainprop2}
Let $\bR_r \in_u \cR_{r,s(r)}$ and write $\bS_r=\rS(\bR_r)$. Then as $r \to \infty$, 
\[
\left(\left(\frac{21}{40 r}\right)^{1/4}  \bR_r~,~\left(\frac{21}{40 r}\right)^{1/4}  \bS_r\right) \convdist \left(\bM,\bM\right)\, \]
in distribution for the Gromov-Hausdorff-Prokhorov topology.
\end{prop}

\vspace*{3mm}

\noindent {\bf Remarks.} 
\begin{enumerate}
\item The proof of \refP{mainprop2}, given in \refS{sec:mass_proj}, uses the convergence of simple quadrangulations, proved in \cite{ABA}, to deduce convergence of $2$-connected quadrangulations, as a stepping stone to proving the joint convergence of \refT{thm2}. 
The results of \cite{ABA} in turn use the ``rerooting invariance trick'' introduced by \citet{LG}, 
together with the convergence of uniform quadrangulations to the Brownian map \cite{LG,Mi}, to deduce convergence for uniform simple quadrangulations. We mention this to emphasize that the results of this paper do not constitute an independent proof of convergence for uniform quadrangulations. 
\item In \cite{ABA} it is also shown that simple triangulations converge to the Brownian map. Using this, the arguments of the current paper could be modified to show joint convergence of uniformly random triangulations and their largest loopless and simple blocks. 
\end{enumerate}
Before sketching our proof, we first describe the combinatorial relations between $\rQ$, $\rR^\bullet(\rQ)$ and $\rS^\bullet(\rQ)$, on which our proofs rely.
\subsection{Bijections for $\rQ$, $\rR$ and $\rS$}\label{sec:mapandquad}
\addtocontents{toc}{\SkipTocEntry}

Suppose we are given only $\rR^\bullet=\rR^{\bullet}(\rQ)$. What additional information is required to reconstruct $\rQ$? 
Similarly, what do we require in addition to $\rS^{\bullet}=\rS^{\bullet}(\rR)$ in order to reconstruct $\rR^{\bullet}$? In each case, the reconstruction requires augmenting the edges with additional data. The reconstruction (equivalently described as {\em decomposition}) procedures which we describe in this section are all either due to Tutte \cite{T} or are obtained by slight variants of his methods. 

When reconstructing $\rR^{\bullet}$ from $\rS^{\bullet}$, this data consists of a $2$-connected quadrangulation for each edge of $\rS^{\bullet}$. When reconstructing $\rQ$ from $\rR^{\bullet}$, we require a {\em sequence} of quadrangulations for each edge of $\rR^{\bullet}$, together with a second, binary sequence whose entries specify how to attach the quadrangulations in the sequence. In both cases, the root edge must be treated slightly differently from the others (in brief, for the root edge we must specify data twice, once for each side of the edge). We now turn to details. 

 A {\em quadrangulation of a $2$-gon} is a rooted map whose unbounded face has degree $2$, with all other faces of degree $4$, rooted such that the unbounded face lies to the left of the root edge. Temporarily write $\cT$ for the set of quadrangulations of $2$-gons. Given a map in $\cT$, merge the two edges incident to the unbounded face to obtain a map in $\cQ$; we call this the natural bijection between $\cT$ and $\cQ$. For $n \ge 3$, it in fact restricts to a bijection between $\cT_n$ and $\cQ_n$. 
Also, $\cT_2$ contains only one element: the map with one edge and two vertices. Recalling that we also view a single edge as a $2$-connected quadrangulation, it follows that $\cT_2=\cQ_2$, and it is convenient to view the natural bijection as associating these two sets with one another. 
 
Let $\rS=(S,uv)$ be a simple quadrangulation. List the vertices of $\rS$ in breadth-first order as $u_1,\ldots,u_n$ and list the edges of $\rS$ as $uv=e_1,\ldots,e_m$, oriented so that the tail precedes the head in breadth-first order. To build a $2$-connected quadrangulation with simple root block $\rS$, proceed as follows (see Figure~\ref{simple}).

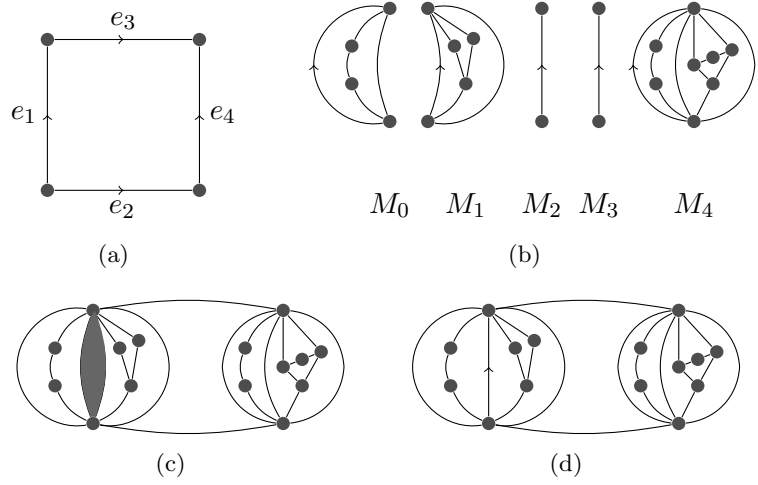
\begin{figure}[htb]

\centering

\subfigure[]{
\begin{tikzpicture}[scale=0.5]
\begin{scope}[decoration={
    markings,
    mark=at position 0.5 with {\arrow{>}}}
    ]
    
	\tikzstyle{main}=[circle,inner sep=0pt,minimum size=5pt,fill=black!70];

	\node[main] (a) at (-2,2) {};
	\node[main] (b) at (2,2){};
	\node[main] (c) at (-2,6){};
	\node[main] (d) at (2,6) {};
	
	\draw[postaction={decorate}] (a) -- node[midway,sloped,right=-0.5pt] [below]{$e_2$} (b);
	\draw[postaction={decorate}] (b) -- node[right]{$e_4$} (d);
	\draw[postaction={decorate}] (a) -- node[left]{$e_1$} (c);
	\draw[postaction={decorate}] (c) -- node[above]{$e_3$} (d);
\end{scope}
\end{tikzpicture}}	
\hspace*{0.5cm}
\subfigure[]{
\begin{tikzpicture}[scale=0.5]
	\tikzstyle{main}=[circle,inner sep=0pt,minimum size=5pt,fill=black!70];
	\begin{scope}[decoration={
    markings,
    mark=at position 0.5 with {\arrow{>}}}
    ]
    	
	\node[main,label={[xshift=0cm,yshift=-1.5cm]$M_0$}] (a) at (-4,-1) {};

\node[main] (a1) at (-5,0) {};
\node[main] (a2) at (-5,1) {};
\node[main] (a3) at (-4,2) {};

\draw (a) to[bend left=20] (a1);
\draw (a1) to[bend left=20] (a2);
\draw (a2) to[bend left=20] (a3);
\draw[postaction={decorate}] (a) to[bend left=50] (-6,0.5) to[bend left=50] (a3);
\draw (a) to[bend left=20] (a3);
	\node[main,label={[xshift=0.5cm,yshift=-1.5cm]$M_1$}] (b) at (-3,-1) {};	

\node[main] (b1) at (-2,0) {};
\node[main] (b2) at (-1.8,1.2){};
\node[main] (b22) at (-2.3,1){};
\node[main] (b3) at (-3,2){};

\draw (b) to[bend right=20] (b1);
\draw (b1) to (b2);
\draw (b2) to (b3);
\draw (b1) to (b22);
\draw (b22) to (b3);
\draw (b) to[bend right=50] (-1,0.5) to[bend right=50] (b3);
\draw[postaction={decorate}]  (b) to[bend right=20] (b3);
	\node[main,label={[xshift=0cm,yshift=-1.5cm]$M_2$}] (c) at (0,-1) {};	

\node[main] (c1) at (0,2) {};

\draw[postaction={decorate}]  (c) to (c1);
	\node[main,label={[xshift=0cm,yshift=-1.5cm]$M_3$}] (d) at (1.5,-1) {};	

\node[main] (d1) at (1.5,2) {};

\draw[postaction={decorate}]   (d) to (d1);

\node[main,label={[xshift=0cm,yshift=-1.5cm]$M_4$}] (e) at (4,-1) {};

\node[main] (e1) at (3,0){};
\node[main] (e2) at (3,1){};
\node[main] (e3) at (4,2){};
\node[main] (e4) at (4,0.5){};
\node[main] (e5) at (4.5,0.7){};
\node[main] (e6) at (5,0.9) {};
\node[main] (e7) at (4.5,0){};
\draw[postaction={decorate}]  (e) to[bend left=40] (2.4,0.5) to[bend left=40] (e3);
\draw (e) to[bend left=20] (e1);
\draw (e1) to[bend left=20] (e2);
\draw (e2) to[bend left=20] (e3);
\draw (e) to[bend left=30] (e3);
\draw (e) to (e7);
\draw (e7) to (e4);
\draw (e7) to (e6);
\draw (e4) to (e5);
\draw (e5) to (e6);
\draw (e4) to (e3);
\draw (e6) to (e3);
\draw (e) to[bend right=40] (5.6,0.5) to[bend right=40] (e3);
\end{scope}
\end{tikzpicture}}	
\\
\subfigure[]{
\begin{tikzpicture}[scale=0.5]
	\tikzstyle{main}=[circle,inner sep=0pt,minimum size=5pt,fill=black!70];
	\begin{scope}[decoration={
    markings,
    mark=at position 0.5 with {\arrow{>}}}
    ]
	
\node[main] (a1) at (-4,0) {};
\node[main] (a2) at (-4,1) {};

\draw (b) to[bend left=20] (a1);
\draw (a1) to[bend left=20] (a2);
\draw (a2) to[bend left=20] (b3);
\draw (b) to[bend left=50] (-5,0.5) to[bend left=50] (b3);
\draw (b) to[bend left=20] (b3);

\node[main] (b) at (-3,-1) {};	
\node[main] (b1) at (-2,0) {};
\node[main] (b2) at (-1.8,1.2){};
\node[main] (b22) at (-2.3,1){};
\node[main] (b3) at (-3,2){};

\draw (b) to[bend right=20] (b1);
\draw (b1) to (b2);
\draw (b2) to (b3);
\draw (b1) to (b22);
\draw (b22) to (b3);
\draw (b) to[bend right=50] (-1,0.5) to[bend right=50] (b3);
\draw (b) to[bend right=20] (b3);

\path[fill=black!60] (b) to[bend left=20] (b3) to (-2.95,2) to[bend left=20] (b);

\node[main] (e) at (2,-1) {};

\node[main] (e1) at (1,0){};
\node[main] (e2) at (1,1){};
\node[main] (e3) at (2,2){};
\node[main] (e4) at (2,0.5){};
\node[main] (e5) at (2.5,0.7){};
\node[main] (e6) at (3,0.9) {};
\node[main] (e7) at (2.5,0){};

\draw (e) to[bend left=40] (0.4,0.5) to[bend left=40] (e3);
\draw (e) to[bend left=20] (e1);
\draw (e1) to[bend left=20] (e2);
\draw (e2) to[bend left=20] (e3);
\draw (e) to[bend left=30] (e3);
\draw (e) to (e7);
\draw (e7) to (e4);
\draw (e7) to (e6);
\draw (e4) to (e5);
\draw (e5) to (e6);
\draw (e4) to (e3);
\draw (e6) to (e3);
\draw (e) to[bend right=40] (3.6,0.5) to[bend right=40] (e3);

\draw (b) to[bend right=10] (e);
\draw (b3) to[bend left=10] (e3);
\end{scope}
\end{tikzpicture}}	
\hspace*{.5cm}
\subfigure[]{
\begin{tikzpicture}[scale=0.5]
	\tikzstyle{main}=[circle,inner sep=0pt,minimum size=5pt,fill=black!70];
	\begin{scope}[decoration={
    markings,
    mark=at position 0.5 with {\arrow{>}}}
    ]
\node[main] (a1) at (-4,0) {};
\node[main] (a2) at (-4,1) {};

\draw (b) to[bend left=20] (a1);
\draw (a1) to[bend left=20] (a2);
\draw (a2) to[bend left=20] (b3);
\draw (b) to[bend left=50] (-5,0.5) to[bend left=50] (b3);

\node[main] (b) at (-3,-1) {};	
\node[main] (b1) at (-2,0) {};
\node[main] (b2) at (-1.8,1.2){};
\node[main] (b22) at (-2.3,1){};
\node[main] (b3) at (-3,2){};

\draw (b) to[bend right=20] (b1);
\draw (b1) to (b2);
\draw (b2) to (b3);
\draw (b1) to (b22);
\draw (b22) to (b3);
\draw (b) to[bend right=50] (-1,0.5) to[bend right=50] (b3);
\draw[postaction={decorate}]   (b) to (b3);

\node[main] (e) at (2,-1) {};

\node[main] (e1) at (1,0){};
\node[main] (e2) at (1,1){};
\node[main] (e3) at (2,2){};
\node[main] (e4) at (2,0.5){};
\node[main] (e5) at (2.5,0.7){};
\node[main] (e6) at (3,0.9) {};
\node[main] (e7) at (2.5,0){};

\draw (e) to[bend left=40] (0.4,0.5) to[bend left=40] (e3);
\draw (e) to[bend left=20] (e1);
\draw (e1) to[bend left=20] (e2);
\draw (e2) to[bend left=20] (e3);
\draw (e) to[bend left=30] (e3);
\draw (e) to (e7);
\draw (e7) to (e4);
\draw (e7) to (e6);
\draw (e4) to (e5);
\draw (e5) to (e6);
\draw (e4) to (e3);
\draw (e6) to (e3);
\draw (e) to[bend right=40] (3.6,0.5) to[bend right=40] (e3);

\draw (b) to[bend right=10] (e);
\draw (b3) to[bend left=10] (e3);
\end{scope}
\end{tikzpicture}}	
\caption{\small (a) A simple quadrangulation. (b) ``Decorations'' for the edges. (c) After attaching the decorations. (d) The map $\rR$.}
\label{simple}
\end{figure}

\begin{enumerate}
\item Create a second copy $e_0$ of the edge $uv$ so that $e_0$ lies to the left of $e_1$. 
\item For $0 \le i \le m$ let $\rM_i$ be a $2$-connected quadrangulation, and let $\rM'_i=(M_i,u_iv_i)$ be the quadrangulation of a $2$-gon associated to $\rM_i$ by the natural bijection.
\item For each $0 \le i \le m$, identify the edge $e_i$ with the root edge $u_iv_i$ of $\rM_i'$. The resulting map has a single facial $2$-cycle (lying between $\rM_0$ and $\rM_1$), with vertices $u$ and $v$; collapse it and root at the resulting edge $uv$.
\end{enumerate}
Call the resulting map $\rR$. Then $\rR$ is a $2$-connected quadrangulation with $\rS^{\bullet}(\rR)=\rS$. We note that
\begin{align}
|e(\rR)|=&~ |e(\rS)|+\sum_{i=0}^{|e(\rS)|} |e(\rM_{i})|~\I{|e(\rM_i)|\ne1}
=-1+\sum_{i=0}^{|e(\rS)|}(1+|e(\rM_i)|~\I{|e(\rM_i)|\ne1})~. \label{eq:edge_count0}
\end{align}

\begin{prop}\label{decom1}
The above procedure induces a bijection $\varphi$ between $\cR$ and the set 
\[
\{ (\rS,\Theta): \rS\in \cS, \Theta \in \cR^{|e(\rS)|+1}\}~.
\]
\end{prop}
\begin{proof}
Given a $2$-connected quadrangulation of a $2$-gon, collapsing the unbounded face to form a single edge (which is equivalent to taking the simple nerve), then rooting at this edge, yields a $2$-connected quadrangulation. This operation is easily seen to be a bijection. In view of the fact that the quadrangulation $\rR\in\cR$ in the above construction has $\rS^{\bullet}(\rR)=\rS$, the result follows. 
\end{proof}

Next let $\rR=(R,uv)$ be a $2$-connected quadrangulation and list $e(\rR)$ as $e_1,\ldots,e_m$, as above. For each integer $1\le i\le m$, write $e_i^+$ and $e_i^-$ for the head and the tail of $e_i$ respectively. To build a quadrangulation with $2$-connected root block $\rR$, proceed as follows (see Figure~\ref{2con}).

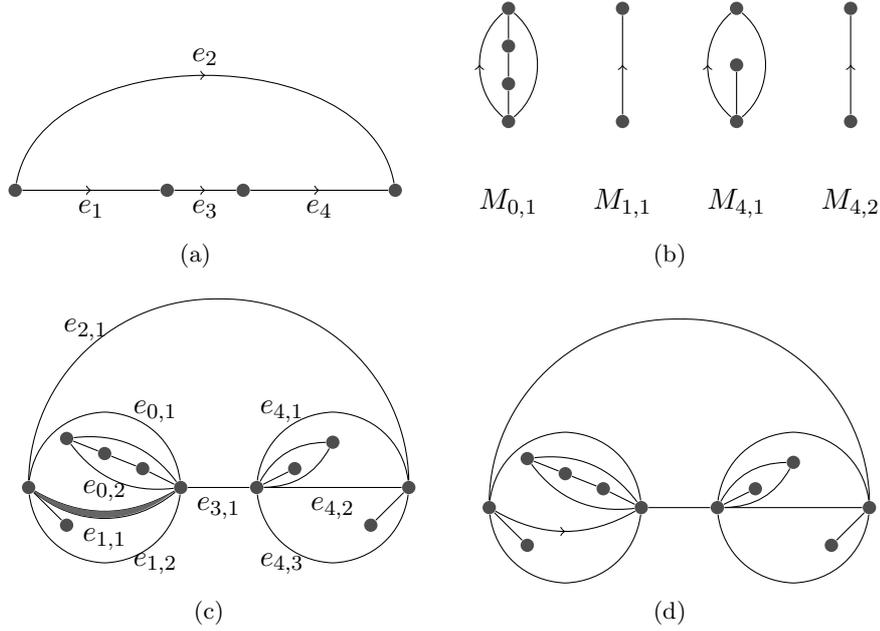
\begin{figure}[htb]

\centering

\subfigure[]{
\begin{tikzpicture}[scale=0.5]
	\tikzstyle{main}=[circle,inner sep=0pt,minimum size=5pt,fill=black!70];
\begin{scope}[decoration={
    markings,
    mark=at position 0.5 with {\arrow{>}}}
    ]
	\node[main] (a) at (-2,4) {};
	\node[main] (b) at (2,4){};
	\node[main] (d) at (4,4){};
	\node[main] (c) at (8,4) {};
	
	\draw[postaction={decorate}]   (a) -- node[below]{$e_1$} (b);
	\draw[postaction={decorate}]   (a) to[bend left=80] node[above]{$e_2$}  (c);
	\draw[postaction={decorate}]   (d) to node[below]{$e_4$} (c);
	\draw[postaction={decorate}]   (b) to node[below]{$e_3$} (d);
\end{scope}
\end{tikzpicture}}	
\hspace*{0.5cm}
\subfigure[]{
\begin{tikzpicture}[scale=0.5]
	\tikzstyle{main}=[circle,inner sep=0pt,minimum size=5pt,fill=black!70];
	\begin{scope}[decoration={
    markings,
    mark=at position 0.5 with {\arrow{>}}}
    ]	\node[main,label={[xshift=0cm,yshift=-1.5cm]$M_{0,1}$}] (a) at (-4,-1) {};

\node[main] (a1) at (-4,0){};
\node[main] (a2) at (-4,1){};
\node[main] (a3) at (-4,2){};
\draw[postaction={decorate}]   (a) to[bend left=50] (a3);
\draw (a) to (a1) to (a2) to (a3);
\draw (a) to[bend right=50] (a3);
	\node[main,label={[xshift=0cm,yshift=-1.5cm]$M_{1,1}$}] (b) at (-1,-1) {};	

\node[main] (b1) at (-1,2){};

\draw[postaction={decorate}]   (b) to (b1);
	\node[main,label={[xshift=0cm,yshift=-1.5cm]$M_{4,1}$}] (c) at (2,-1) {};	

\node[main] (c1) at (2,0.5){};
\node[main] (c2) at (2,2){};
\draw[postaction={decorate}]   (c) to[bend left=50] (c2);
\draw (c) to (c1);
\draw (c) to[bend right=50] (c2);
	\node[main,label={[xshift=0cm,yshift=-1.5cm]$M_{4,2}$}] (d) at (5,-1) {};	

\node[main] (d1) at (5,2) {};

\draw[postaction={decorate}]   (d) to (d1);
\end{scope}
\end{tikzpicture}}	
\\
\subfigure[]{
\begin{tikzpicture}[scale=0.5]
	\tikzstyle{main}=[circle,inner sep=0pt,minimum size=5pt,fill=black!70];

	\node[main] (a) at (-2,4) {};
	\node[main] (b) at (2,4){};
	\node[main] (d) at (4,4){};
	\node[main] (c) at (8,4) {};
	\node[main] (b1) at (1,4.5){};
	\node[main] (b2) at (0,4.9){};
	\node[main] (b3) at (-1,5.3){};
	\node[main] (a1) at (-1,3){};
	
	\draw (a) to[bend left=45] node[above]{$e_{2,1}$} (3,9) to[bend left=45] (c);
	\draw (c) to node[below]{$e_{4,2}$} (d);
	\draw (a) to[bend left=40] (0,6) to[bend left=40] node[above]{$e_{0,1}$}  (b);
	\draw (b) to (b1) to (b2) to (b3);
	\draw (b) to[bend left=30] (b3);
	\draw (b) to[bend right=30] (b3);
	\draw (b) to node[below]{$e_{3,1}$} (d);

	\draw (a) to[bend right=30] node[above]{$e_{0,2}$} (b);
	\draw (a) to[bend right=40] node[below]{$e_{1,1}$} (b);
	\draw (a) to (a1);
	\draw (a) to[bend right=40] (0,2) to[bend right=40] node[below]{$e_{1,2}$}  (b);

	\node[main] (d1) at (7,3){};
	\draw (c) to (d1);
	\node[main] (c1) at (5,4.5){};
	\node[main] (c2) at (6,5.2){};
	\draw (d) to (c1);
	\draw (d) to[bend left=40] node[above]{$e_{4,1}$} (6,6) to[bend left=40]   (c);
	\draw (d) to[bend right=40] node[below]{$e_{4,3}$} (6,2) to[bend right=40]  (c);
	\draw (d) to[bend left=30] (c2);
	\draw (d) to[bend right=30]  (c2);

	\path[fill=black!60] (a) to[bend right=30] (b) to[bend left=40] (a);

\end{tikzpicture}}	
\hspace*{.5cm}
\subfigure[]{
\begin{tikzpicture}[scale=0.5]
	\tikzstyle{main}=[circle,inner sep=0pt,minimum size=5pt,fill=black!70];
\begin{scope}[decoration={
    markings,
    mark=at position 0.5 with {\arrow{>}}}
    ]
	\node[main] (a) at (-2,4) {};
	\node[main] (b) at (2,4){};
	\node[main] (d) at (4,4){};
	\node[main] (c) at (8,4) {};
	\node[main] (b1) at (1,4.5){};
	\node[main] (b2) at (0,4.9){};
	\node[main] (b3) at (-1,5.3){};
	\node[main] (a1) at (-1,3){};
	
	\draw (a) to[bend left=45]  (3,9) to[bend left=45] (c);
	\draw (c) to (d);
	\draw (a) to[bend left=40] (0,6) to[bend left=40]  (b);
	\draw (b) to (b1) to (b2) to (b3);
	\draw (b) to[bend left=30] (b3);
	\draw (b) to[bend right=30] (b3);
	\draw (b) to (d);

	\draw[postaction={decorate}]   (a) to[bend right=30]   (b);
	\draw (a) to (a1);
	\draw (a) to[bend right=40] (0,2) to[bend right=40] (b);

	\node[main] (d1) at (7,3){};
	\draw (c) to (d1);
	\node[main] (c1) at (5,4.5){};
	\node[main] (c2) at (6,5.2){};
	\draw (d) to (c1);
	\draw (d) to[bend left=40] (6,6) to[bend left=40]   (c);
	\draw (d) to[bend right=40]  (6,2) to[bend right=40]  (c);
	\draw (d) to[bend left=30] (c2);
	\draw (d) to[bend right=30]  (c2);

\end{scope}
\end{tikzpicture}}	
\caption{\small The quadrangulation in (d) can be reconstructed from its $2$-connected core in (a) with the decoration $((L_i,b_i): 0\le i\le r)$ where $L_0=(M_{0,1}),b_0=(1),L_1=(M_{1,1}),b_1=(0)$, $L_2=L_3=\emptyset,b_2=b_3=\emptyset,L_4=(M_{4,1},M_{4,2}),b_4=(0,1)$.}
\label{2con}
\end{figure}

\begin{enumerate}
\item Create a second copy $e_0$ of the edge $uv$ so that $e_0$ lies to the left of $e_1$. 
\item For $0 \le i \le m$ fix $\ell_i \in \N_{\ge 0}$ and sequences $L_i=(\rM_{i,j}:1 \le j \le \ell_i) \in \cQ^{\ell_i}$, 
$b_i = (b_{i,j}:1 \le j \le \ell_i) \in \{0,1\}^{\ell_i}$. 
\item For each $1 \le i \le m$, add an additional $\ell_i$ copies of $e_i$; label the resulting $\ell_i+1$ copies of $e_i$ as $e_{i,1},\ldots,e_{i,\ell_i+1}$ in clockwise order around $e_i^-$. 
\item For $0 \le i \le m$ and $1 \le j \le \ell_i$, let $\rM_{i,j}'$ be the quadrangulation of a $2$-gon associated to $\rM_{i,j}$ by the natural bijection. 
\item Attach $\rM_{i,j}'=(M_{i,j},u_{i,j}v_{i,j})$ inside the $2$-cycle formed by $e_{i,j}$ and $e_{i,j+1}$ by identifying $u_{i,j}$ with $e_i^-$ (if $b_{i,j}=0$) or $e_i^+$ (if $b_{i,j}=1$). 
The resulting map has a single facial $2$-cycle, with edges $e_{0,\ell_0+1}$ and $e_{1,1}$; collapse it and root at the resulting edge $uv$.
\end{enumerate}
Call the resulting map $\rQ$. Then $\rQ$ is a connected quadrangulation with $\rR^{\bullet}(\rQ)=\rR$. 
We note that
\begin{align}
|e(\rQ)|& =|e(\rR)|+\sum_{i=0}^{|e(\rR)|} \sum_{j=1}^{\ell_i} (|e(\rM_{i,j})|+1+\I{|e(\rM_{i,j})|\ne1})) \notag\\
 &=  -1+\sum_{i=0}^{|e(\rR)|}\bigg(1+ \sum_{j=1}^{\ell_i}( |e(\rM_{i,j})|+1+\I{|e(\rM_{i,j})|\ne1})\bigg)~. \label{eq:edge_count}
\end{align}
In the following proposition we write $(\cQ\times\{0,1\})^* = \{\emptyset\} \cup \bigcup_{n \in\N} (\cQ\times\{0,1\})^n$. 
\begin{prop}\label{decom2}
The above procedure induces a bijection $\psi$ between $\cQ$ and the set 
\[
\left\{ (\rR,\Gamma): \rR\in \cR, \Gamma \in \big((\cQ \times \{0,1\})^*\big)^{|e(\rR)|+1}\right\}~.
\]
\end{prop}
\begin{proof}
This is immediate from the fact that the above construction has $\rR^{\bullet}(\rQ)=\rR$.
\end{proof}

For both decompositions, we refer informally to the maps in the vectors $\Theta$ and $\Gamma$ as {\em decorations} or as {\em pendant submaps}.

\subsection{Proof sketch for \refT{thm2}} \label{sec:sketch}
\addtocontents{toc}{\SkipTocEntry}

In this subsection, we assume familiarity with the Gromov-Hausdorff and Gromov-Hausdorff-Prokhorov distances. The relevant definitions appear in \refS{sec:pre}. We begin by stating (and sketching the proof of) a joint convergence result for a 2-connected quadrangulation and its largest simple block; the proof of this result contains most of the key ideas for the proof of Theorem~\ref{thm2}. 

Given $\bR_r=(R_r,{e}_r) \in_u \cR_{r,s(r)}$, it is easily seen that $\bS_r=\rS(\bR_r)$ is uniformly distributed over $\cS_{s(r)}$. Then by \citep[Theorem 1]{ABA}, $(3/8s(r))^{1/4} \bS_r \convdist \bM$ as $s(r)\to\infty$. Also, the definition of $s(r)$ guarantees that 
$(\frac 3{8s(r)})^{1/4} \cdot (\frac {21}{40 r})^{-1/4}\to 1$ as $r\to\infty$. 

Let ${e}'$ be the $\prec_{\bR_r}$-minimal oriented edge of $\bS_r$. 
If ${e}_r \in e(\bS_r)$ then $\bS_r = \rS^{\bullet}(\bR_r)$. 
Write $\bR_r'=(R_r,{e}')$. By \refP{decom1}, $\bR_r'$  uniquely decomposes as $\varphi(\bR_r')=(S,\Theta)\in \cS_{s(r)} \times \cR^{|e(\rS)|+1}$, and our choice of ${e}'$ guarantees that $S=\bS_r$. Write $\Theta=(\Theta_i:0 \le i \le 2s(r)-4)$, and 
\begin{align*}
L(\bR_r) & = \max\left\{|v(\Theta_i)|: 0\le i\le 2s(r)-4 \right\}~,\\
D(\bR_r) & = \max\left\{\diam(\Theta_i): 0\le i\le 2s(r)-4 \right\}~. 
\end{align*}
In words, $L(\bR_r)$ and $D(\bR_r)$ are the greatest number of vertices and the greatest diameter, respectively, of any submap pendant to the biggest simple block of $\bR_r$. The identification of $\bS_r$ as a submap of $\bR_r$ gives the bound
$
\dgh(\bR_r,\bS_r) \le D(\bR_r).
$
To prove that $ \left(\frac {21}{40 r}\right)^{1/4}\dgh(\bR_r,\bS_r) = o(1)$ in probability, it thus suffices to show that 
$
 \left(\frac {21}{40 r}\right)^{1/4}D(\bR_r)=o(1)
$ in probability.
(Note that here we have the Gromov-Hausdorff rather than Gromov-Hausdorff-Prokhorov distance!) 

To accomplish this, we use the methodology developed by \citet{BFSS}, which allows one to describe the largest block size of a map whenever the map may be described by a recursive decomposition into rooted blocks, using a suitable composition schema; this is explained in greater detail in \refS{sec:schema}. 
We thereby obtain the following distributional result for $|v(\bS_r)|$.

\begin{prop}\label{airyprop2}
Let $\bR_r\in_u \cR_r$, then for any $A>0$, uniformly over $x\in [-A,A]$, 
\[
\p{\sb(\bR_r) = \lfloor 5 r/7+ x r^{2/3} \rfloor} = \frac {\beta\cA\left(\beta x\right)}{r^{2/3}} ( 1+o(1))~,
\]
where $\beta$ is given in \refP{schema2}.
\end{prop}
  
The proof of \refP{airyprop2} appears in \refS{sec:schema}. The range of values for $r$ in the above local limit theorem is what leads to our choice for the range of $s(r)$ in \refT{thm2} and \refT{thm:blocksizes}. The following proposition bounds the size of the largest simple block of a random $2$-connected quadrangulation.

\begin{prop}\label{pendant2}
For any $A>0$, there exist positive constants $c_1$ and $c_2$ such that for all $r\in\N$ and for integer $k\in\left(5 r/7 +A r^{2/3}, r\right]$, if $\bR_r\in_u \cR_r$,
\[
\p{\sb(\bR_r) = k} \le c_1 \exp\left(-c_2 r \left(k/r - 5/7\right)^3\right)~.
\]
\end{prop}

This proposition is a slight extension of \citep[Theorem 1]{BFSS}, which proves similar bounds but requires that $(r-k)/r^{2/3} \to \infty$. We do not reprove the entire result, but simply analyze the behaviour in the range not covered in the work of \cite{BFSS}. We use \refP{pendant2} in proving stretched exponential tail bounds for the size of the largest pendant submap in a random $2$-connected quadrangulation.

\begin{prop}\label{secondlargest}
For all $\veps\in(0,1/3)$, there exist positive constants $c_1$, $c_2$, and $c_3=c_3(\veps)$ such that for all $r\in\N$, if $\bR_r\in_u \cR_{r,s(r)}$,
\[
\p{L(\bR_r) \ge r^{2/3+\veps}}\le c_1  \exp\left(-c_2 r^{c_3}\right)~.
\]
\end{prop}

Proofs for \refP{pendant2} and \refP{secondlargest} are given in \refS{sec:complement}.

Next we deduce a bound for $D(\bR_r)$ through extending a result by \citet{CS}. The following proposition follows straightforwardly from \citep[Proposition 4]{CS}.

\begin{prop}{\em (\cite{CS}).}\label{citediam}
 There exist positive constants $y_0$, $C_1$, and $C_2$ such that for all $y> y_0$ and $q\in\N$, if $\bQ_q \in_u \cQ_q$, 
\[
\p{\diam\left(\bQ_q\right) > y q^{1/4} } \le C_1  \exp(-C_2 y)~.
\]
\end{prop}

This bound is for connected quadrangulations rather than $2$-connected ones. However, at the cost of polynomial corrections, we are able to transfer the result to the latter family of quadrangulations, as shown in \refS{sec:diambound}. This in particular yields the following bound. 

\begin{prop}\label{diam}
Let $\bR_r\in_u \cR_{r,s(r)}$, then there exist positive constants $C_1$, $C_2$, and $C_3$ such that 
\[
\p{D(\bR_r) \ge  r^{5/24}}\le C_1 \exp\left(-C_2  r^{C_3}\right)~.
\]
\end{prop}

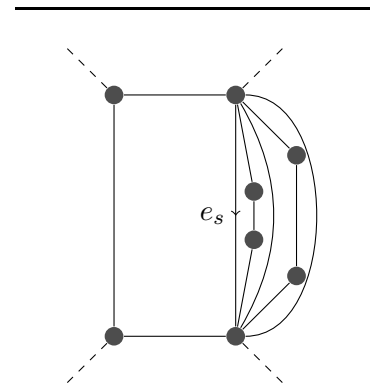
\begin{wrapfigure}[18]{R}{.31\textwidth}
\rule[5pt]{.31\textwidth}{0.5pt}
\centering
\begin{tikzpicture}[scale=0.8]
	\tikzstyle{main}=[circle,inner sep=0pt,minimum size=7pt,fill=black!70,font=\small];
	\begin{scope}[decoration={
    markings,
    mark=at position 0.5 with {\arrow{>}}}
    ]
	\node[main] (a) at (0,4) {} ;
	\node[main] (b) at (0,0) {};
	\node[main] (c) at (-2,0) {};
	\node[main] (d) at (-2,4) {};
	\node (e) at (1,5) {};
	\node (h) at (1,-1) {};
	\node (f) at (-3,-1) {};
	\node (g) at (-3,5) {};
	
	\node[main] (1) at (0.3,2.4) {};
	\node[main] (2) at (0.3,1.6) {};
	\node[main] (21) at (1,3) {};
	\node[main] (22) at (1,1) {};
	
    \draw[postaction={decorate}]  (a) to node[left]{${e_s}$} (b);	
    \draw  (b) edge (c);	
    \draw  (d) edge (c);	
    \draw  (a) edge (d);	
    \draw (a) edge [dashed] (e);
    \draw (b) edge [dashed] (h);
    \draw (c) edge [dashed] (f);
    \draw (d) edge [dashed] (g);
    
    \draw  (a) edge (1);	    
    \draw  (1) edge (2);	
    \draw  (2) edge (b);	
    \draw  (a) edge [bend left=30] (b);	
    \draw  (a) edge (21);	    
    \draw  (21) edge (22);	
    \draw  (22) edge (b);	
    \draw  (a) edge [bend left=90] (b);	
\end{scope}
\end{tikzpicture}
\caption{\small Parallel alternating $1$-paths and $3$-paths attached to the root edge ${e_s}$.}
\label{fig:pointmass}
\rule[8pt]{.31\textwidth}{0.5pt} 
\end{wrapfigure}

The above results immediately give rise to Gromov-Hausdorff convergence of $\left( \bR_r,\bS_r\right)$ after rescaling, as shown in \refP{ghprop2} in the end of \refS{sec:diambound}. However, to deduce GHP convergence, the above propositions are insufficient, as they do not guarantee that the uniform measures on $v(\bR_r)$ and $v(\bS_r)$ are close in the Prokhorov sense. Here is an example of the sort of issue that may {\em a priori} still occur. For all $s\in\N$, let $\bS_s\in_u\cS_s$ have root edge ${e_s}$. Let $P_s$ be the quadrangulation of a $2$-gon with $2\lfloor s/5\rfloor+2$ vertices composed of parallel alternating $1$-paths and $3$-paths, and write ${e'_s}$ for one of the boundary edges of $P_s$. Then identify ${e_s}$ with ${e'_s}$, embed $P_s$ in the face of $\bS_s$ to the left of ${e_s}$, and write $\bR_s'$ for the resulting quadrangulation; see \refF{fig:pointmass}. Recall that $\bM=(M,d,\mu)$ is the Brownian map. 
Then it is not hard to see that $\left(\bR_s',\bS_s\right)$ converges after rescaling to $(\bM',\bM)$, where $\bM'=(M,d,\mu')$ has the geometry of the Brownian map but has mass measure $\frac57\mu+\frac27 \delta_\rho$, where $\rho$ is a point of $M$ with law $\mu$.

To prevent the masses of ``pendant submaps" from concentrating on small regions in this manner, we use that they attach to exchangeable random locations on the simple block and that each of them has asymptotically negligible size. The first follows from the details of the construction of a $2$-connected quadrangulation from its simple root block, explained in \refS{sec:mapandquad}; the second is a consequence of \refP{secondlargest}.

In order to show that the facts from the preceding paragraph suffice to imply joint convergence, we prove a general result on the preservation of Gromov-Hausdorff-Prohkorov convergence under small random modifications; our result relies on results of Aldous on concentration for partial sums of exchangeable random variables. Details for this part of the proof appear in Sections~\ref{sec:exch_measures} and~\ref{sec:mass_proj}.

We conclude the proof sketch by explaining how we strengthen \refP{mainprop2} to prove \refT{thm2}.
First, with $\bQ_q\in_u\cQ_{q,r(q),s(q)}$, we show that $\rR(\bQ_q)$ contains $\rS(\bQ_q)$ with high probability. The joint convergence of the second and third coordinates in \refT{thm2} then follows from \refP{mainprop2}. 

The convergence of the first coordinate does not follow from the existing result by \citet{LG} or \citet{Mi}, because $\bQ_q$ here is not uniformly distributed over $\cQ_q$, but conditioned on $\b(\bQ_q) = r(q)$ and $\sb(\bQ_q)=s(q)$. To deal with this, we require versions of Propositions \ref{secondlargest} and \ref{diam} that apply to uniform quadrangulations sampled from $\cQ_q$ and $\cQ_{q,r(q),s(r(q))}$. These follows straightforwardly once we show that with high probability, $\rS(\rR(\bQ_q))=\rS(\bQ_q)$. We postpone the details.

A reprise of the argument for \refP{mainprop2} then shows that if $\bQ_q\in_u \cQ_{q,r(q),s(q)}$, then $(\bQ_q,\rR(\bQ_q))\convdist (\bM,\bM)$ as $q\to\infty$. Since we also know $(\rR(\bQ_q), \rS(\bQ_q))\convdist (\bM,\bM)$ as $q\to\infty$, \refT{thm2} follows immediately.

\section{Preliminaries}\label{sec:pre}

\subsection{Hausdorff and Prokhorov distances}
\addtocontents{toc}{\SkipTocEntry}

Let $(V,d)$ be a compact metric space with its Borel $\sigma$-algebra $\cB(V)$. Given $A\subset V$, the {\em $\veps$-neighborhood} of $A$ is defined as 
\[
A^\veps = \left\{x\in V: \exists y\in A, d(x,y)<\veps\right\}~.
\]
The {\em Hausdorff distance} $d_\rH$ between two non-empty subsets $X,Y$ of $(V,d)$ is defined as
\[
d_\rH(X,Y) = \inf\left\{\veps>0:X\subset Y^\veps,Y\subset X^\veps\right\}~.
\]

Denote by $\cP(V)$ the collection of all probability measures on the measurable space $(V,\cB(V))$. The {\em Prokhorov distance} $d_\rP : \cP(V)^2 \to [0,\infty)$ between two Borel probability measures $\mu$ and $\nu$ on $V$ is given by
\[
d_\rP(\mu,\nu) = \inf\left\{\veps>0: \mu(A)\le \nu(A^\veps) + \veps \mbox{ and } \nu(A)\le \mu(A^\veps)+\veps,\forall A\in \cB(V)\right\}~.
\]

\subsection{Gromov-Hausdorff(-Prokhorov) distance}\label{sec:ghpintro}
\addtocontents{toc}{\SkipTocEntry}

We refer the reader to \cite{BBI} and \cite{Mi,LG} for more details on the Gromov-Hausdorff and Gromov-Hausdorff-Prokhorov distances and the topologies they induce. 
Let $(V,d)$ and $(V',d')$ be two compact metric spaces. A {\em correspondence} between $V$ and $V'$ is a set $C\subset V\times V'$ such that for every $x\in V$, there is $x'\in V'$ with $(x,x')\in C$, and vice versa. We write $C(V,V')$ for the set of correspondences between $V$ and $V'$. The {\em distortion} of any set $C\subset V\times V'$ with respect to $d$ and $d'$ is given by 
\[
\dis\left(C;d,d'\right) = \sup\left\{|d(x,y) - d'(x',y')|: (x,x')\in C,  (y,y')\in C\right\}~.
\]
The {\em Gromov-Hausdorff distance} between $(V,d)$ and $(V',d')$ is defined as
\[
\dgh\left((V,d),(V',d')\right) = \inf\left\{\veps>0: \exists C\in C(V,V'),\dis(C;d,d')\leq 2\veps\right\}~.
\]

Next, suppose $\mu$ and $\mu'$ are non-negative Borel measures on $(V,d)$ and $(V',d')$, respectively.
The {\em Gromov-Hausdorff-Prokhorov} distance between $\rV=(V,d,\mu)$ and $\rV'=(V',d',\mu')$ is given by
\[
\dghp(\rV,\rV') = \inf\left[ \max\left\{ d_\rH(\phi(V),\phi'(V')) , d_\rP(\phi_*\mu,\phi'_* \mu') \right\}\right]
\]
where the infimum is taken over all isometries $\phi,\phi'$ from $(V,d),(V',d')$ into a metric space $(Z,\delta)$ (see \citet[Section 6.2]{Mi09}).
Writing $\K$ for the set of all isometry classes of compact measured metric spaces, $(\K,\dghp)$ is a Polish space; when we refer to GHP convergence we mean convergence in this space. 

\subsection{The Airy distribution}
\addtocontents{toc}{\SkipTocEntry}

The {\em Airy distribution} is the probability distribution whose density is
\begin{align*}
\cA(x) =& 2 e^{-2x^3/3} \left(x \mbox{Ai}(x^2)-\mbox{Ai}'(x^2)\right)\\
=&~ \frac 1 {\pi x} \sum_{n\in\N} (-x 3^{2/3})^n \frac{\Gamma(2n/3+1)}{n!} \sin (-2n\pi/3)~,
\end{align*}
where the Airy function $\mbox{Ai}$ is given by
\begin{align*}
\mbox{Ai}(z)=& \frac 1 {2\pi} \int^\infty_{-\infty} e^{i(zt+t^3/3)} dt\\
=&~ \frac 1 {\pi 3^{2/3}} \sum_{n\in\N_{\ge0}} (3^{1/3} z)^n  \frac{\Gamma((n+1)/3)}{n!} \sin (2(n+1)\pi/3)~.
\end{align*}

\section{Composition schemata}
\label{sec:schema}
Let $F(z) = \sum_{n \in\N_{\ge0}} F_n z^n$ be a generating function (i.e. an analytic function with nonnegative integer coefficients) with positive radius of convergence $r=r_F$. We say $F$ is {\em singular with exponent $3/2$} if the following properties hold. 
\begin{itemize}
\item There exists $\eps > 0$ such that $F$ is continuable in $\Delta = \{z: |z| < r+\eps, z \not\in [r,r+\eps)\}$.
\item There exist positive constants $a=a_F,b=b_F,c=c_F$ such that $F(z) = a-b(1-z/r) + c(1-z/r)^{3/2} + O((1-z/r)^2)$ as $z \to r$ in $\Delta$. 
\end{itemize}

Gao and Wormald \cite{GW} derived asymptotics for the coefficients of $F$ under the above conditions. 
\begin{prop}[\cite{GW}, Theorem 1 (iii)]\label{expansionquote}
Let $F$ be singular with exponent $3/2$, let $r$ and $c$ be as above. Then 
\[
F_n \sim \frac{3c}{4\pi^{1/2}}\frac{r^{-n}}{n^{5/2}}. 
\]
\end{prop}

Next, let $C$ and $H$ be generating functions with positive coefficients, and define a bivariate generating function $M$ by $M(z,u)=C(uH(z))$; \citet{BFSS} call this a {\em composition schema}. We generically write $C_k = [z^k] C(z)$ and $M_n = [z^n] M(z,1)$, and for $n \in\N$ let $X_n$ be a real random variable with law given by 
\[
\p{X_n=k} = \frac{C_k}{M_n} [z^n] H(z)^k. 
\]
We quote from \cite{BFSS}:
\begin{quote}
Combinatorially, this corresponds to a composition $\mathcal{M}=\mathcal{C} \circ \mathcal{H}$ between classes of [rooted] objects, where objects of type $\mathcal{H}$ are substituted freely at individual ``atoms'' (e.e., nodes, edges, or faces) of elements of $\mathcal{C}$... $[z^n u^k] M(z,u)$ gives the number of $\mathcal{M}$-objects of total size $n$ whose $\mathcal{C}$-component (the ``core'') has size $k$, and $X_n$ is the corresponding random variable describing core-size in this general context. 
\end{quote}
More precisely, $X_n$ is the law of the size of the $\mathcal{C}$-component containing the root, in an object chosen uniformly at random from among all $\mathcal{M}$-objects of size $n$. The connection with the bijections for quadrangulations described in \refS{sec:mapandquad} should be clear. We say the triple $(M,C,H)$ is a {\em map schema} if $C$ and $H$ are both singular with exponent $3/2$ and 
additionally $H(r_H)=r_C$.\footnote{In \cite{BFSS}, this is called a {\em critical composition schema of singular type $(\frac{3}{2} \circ \frac{3}{2})$}. We shorten this to ``map schema'' as such schemata seem to primarily arise in the study of maps.}  The following results are all from \cite{BFSS}.
\begin{prop}[\cite{BFSS}, Theorems 1 and 5]\label{quoteprop}
Suppose $(M,C,H)$ is a map schema with
\[
C(z) = c_0 - c_1 (1- z/r_C) + c_{3/2} (1-z/r_C)^{3/2} + O\left( (1-z/r_C)^2 \right)~,
\]
\[
H(z) = h_0 - h_1 (1- z/r_H) + h_{3/2} (1-z/r_H)^{3/2} + O\left( (1-z/r_H)^2 \right)~,
\]
the expansions for $C(z)$ and $H(z)$ valid in some neighbourhoods of $r_C$ and of $r_H$, respectively. 
Let $\alpha=\alpha_{(M,C,H)}, \beta=\beta_{(M,C,H)}$ and $\gamma=\gamma_{(M,C,H)}$ be defined by 
\[
\alpha = \frac{c_1 h_{3/2} h_0^{1/2} + c_{3/2} h_1^{3/2}}{h_0^{3/2}}, \beta=\frac{h_1^{ 5/3}}{(3 h_{3/2})^{2/3} h_0}, \gamma = \beta  \cdot \frac{c_{3/2}h_1^{3/2}}{\alpha \cdot h_0^{3/2}}\, .
\]
Then 
\[ 
[z^n] M(z,1) \sim \frac{3\alpha }{4\pi^{1/2}}\frac{r_H^{-n}}{n^{5/2}}\, . 
\] 
Furthermore, writing $\alpha_0 = \alpha_{0,(M,C,H)}= h_0/h_1$, for any $A > 0$, 
\begin{equation}\label{eq:central_bound}
\lim_{n \to \infty} \sup_{x \in [-A,A]} \left\vert n^{2/3} \p{X_n = \lfloor \alpha_0 n+xn^{2/3}\rfloor} - \gamma \mathcal{A}(\beta x)\right\vert = 0\, .
\end{equation}
Finally, there exist continuous functions $f:[\alpha_0,1] \to (0,\infty)$ and $g:[\alpha_0,1] \to (0,\infty)$ such that for any function $\lambda:\N \to \N$ with $\lambda(n) \to \infty$, 
\begin{equation}\label{eq:right_bound}
\p{X_n=k} =(1+o(1)) f(k/n)\frac{(k/n - \alpha_0)^{1/2}}{n^{1/2}(1-k/n)^{3/2}} e^{-n(k/n-\alpha_0)^3 g(k/n)}, 
\end{equation}
the preceding asymptotic holding uniformly in $\alpha_0 n + n^{2/3} \lambda(n) < k < n-n^{2/3} \lambda(n)$. 
\end{prop}
\noindent {\bf Remark.} In \cite{BFSS}, schema of the form $M(z,u) = C(uH(z))+D(z)$ are also considered. Replacing $M(z,u)$ by $M(z,u)-D(z)$ turns this into a compositional schema; if the latter is a map schema then Proposition~\ref{quoteprop} applies. The equation involving $D$ is convenient when considering map families in which the core may have size zero; such families should be counted by $[u^0]M(z,u)$, which is identically zero in $C(uH(z))$. 

\begin{cor}\label{quotecor}
Suppose $(M,C,H)$ is a map schema, and let $\alpha_0=\alpha_{0,(M,C,H)}$, $f$ and $g$ be as in \refP{quoteprop}. Then for any function $\lambda:\mathbb{N} \to \mathbb{N}$ with $\lambda(n) \to \infty$ and any $a > 0$, 
\[
\p{X_n = k} = \Theta(1) \cdot \frac{(k/n - \alpha_0)^{1/2}}{n^{1/2}(1-k/n)^{3/2}} e^{-n(k/n-\alpha_0)^3 g(k/n)}\, ,
\]
uniformly over integers $k \in [\alpha_0 n + a n^{2/3},n-\lambda(n)n^{2/3})$. 
\end{cor}
\begin{proof}
Note that if $k=\alpha_0 n + cn^{2/3}$ then 
\[
\frac{(k/n - \alpha_0)^{1/2}}{n^{1/2}(1-k/n)^{3/2}} e^{-n(k/n-\alpha_0)^3 g(k/n)} = \frac{c^{1/2}}{n^{2/3}(1-k/n)^{3/2}} e^{-c^3 g(\alpha_0+c/n^{1/3})}. 
\]
For $|k-\alpha_0 n| = O(n^{2/3})$, the latter is $\Theta(n^{-2/3})$. By (\ref{eq:central_bound}), we indeed have $\p{X_n=k}=\Theta(n^{-2/3})$ for such $k$. 

If the claim of the corollary fails then there exists a sequence $(n_i,i \ge 1)$ and $k_i \in [\alpha_0 n_i + an^{2/3},n_i-\lambda(n_i) n_i^{2/3}]$
along which the ratio of $\p{X_{n_i} = k_i}$ and 
\[
\frac{(k_i/n_i - \alpha_0)^{1/2}}{n_i^{1/2}(1-k_i/n_i)^{3/2}} e^{-n_i(k_i/n_i-\alpha_0)^3 g(k_i/n_i)}
\]
either diverges or tends to zero. By passing to a subsequence if necessary, we may assume that either $k_i - \alpha_0 n_i =O(n^{2/3})$ or $(k_i - \alpha_0 n_i)/n_i^{2/3} \to \infty$. In view of the above computation, the first possibility is in contradiction with (\ref{eq:central_bound}). The second possibility is in contradiction with (\ref{eq:right_bound}); thus neither can occur.  
\end{proof}

\begin{cor}[\cite{BFSS}, Theorem 7]\label{largestcor}
Suppose $(M,C,H)$ is a map schema with $\alpha_0=\alpha_{0,(M,C,H)}$ and $\beta=\beta_{(M,C,H)}$ defined in \refP{quoteprop}. Let $X_n^*$ be the size of the largest  $\mathcal{C}$-component in a random $\mathcal{M}$-map of size $n$ with uniform distribution. Then
\[
\p{X_n^*=\left\lfloor \alpha_0 n + x n^{2/3}\right\rfloor} = \frac{\beta \cA(\beta x)}{n^{2/3}} ( 1 + o(1))\, ,
\]
uniformly for $x$ in any bounded interval.
\end{cor}

Let $M(z),C(z),B(z)$ be the generating functions of rooted connected, $2$-connected, and simple quadrangulations respectively. 
More precisely, we take $[z^n]M(z)=|\cQ_{n+2}|$, $[z^n]C(z)=|\cR_{n+2}|$, and $[z^n]B(z)=|\cS_{n+2}|$ for $n \ge 1$, and 
$[z^n]M(z)=[z^n]C(z)=[z^n]B(z)=0$ for $n=0$. (The latter is slightly at odds with our convention of viewing a single edge as a $2$-connected quadrangulation, but is algebraically convienent below.) Define
\begin{equation}\label{Hdfn}
H(z) = z \left(\frac 1 {1- 2z(1+ M(z)) }\right)^2~,
\end{equation}
\begin{equation}\label{Udfn}
U(z) = z \left( 1+ C(z)\right)^2~.
\end{equation}

The following two lemmas follow immediately from Propositions~\ref{decom1} and~\ref{decom2} respectively.

\begin{lem}\label{H}
We have the following substitution relation between $M(z)$ and $C(z)$:
\begin{equation}\label{MCH}
M(z) = C(H(z))\cdot \frac 1 {1- 2z(1+ M(z)) }+ \frac {2z(1+M(z))} {1- 2z(1+ M(z))}~.
\end{equation}
Equivalently,
\begin{equation}\label{MCH2}
M(z) = C(H(z))+ 2z(1+M(z))^2~.
\end{equation}
\end{lem}

Equation (\ref{MCH2}) is obtained by multiplying both sides of (\ref{MCH}) by $1-2z(1+M(z))$ and then rearranging elements. To see that (\ref{MCH2}) gives a composition schema, we can rewrite it as $\hat{M}(z) = C(H(z))$ with $\hat{M}(z)= M(z)-2z(1+M(z))^2$.

We now take a closer look at equation (\ref{Hdfn}), which describes the ``$M$-decorations'' of an edge of a $C$-object (i.e. of a $2$-connected map). This is best understood with the bijection from Proposition~\ref{decom1} at hand. In the term $2z(1+M(z))$, the multiplier $2$ counts the choice of extremity at which the decoration is attached; $M(z)$ counts the case when attachment is a quadrangulation with at least $3$ vertices (recalling that $z$ marks the number of vertices less two, and the lowest power term of $M(z)$ is $2z$); the additive term $1$ counts the case when attachment is a single edge; the multiplier $z$ adjusts the counting of extra vertices resulting from the attachment (we multiply by $z$ instead of $z^2$ because the attachment vertex is already counted in the core). Taking the reciprocal of $1-2z(1+M(z))$ accounts for the fact that we can attach a sequence of submaps (each two separated by an edge). Squaring the reciprocal accounts for the fact that in a quadrangulation $Q$ we have $|e(Q)|=2(|v(Q)|-2)$. 

In equation (\ref{MCH}), the term $\frac {2z(1+M(z))} {1- 2z(1+ M(z))}$ takes into consideration when the root block is a single edge. The multiplication $\frac 1 {1- 2z(1+ M(z)) }$ in the first term accounts for the extra submap attachment due to split of the root edge (recall the construction preceding \refP{decom1}).

\begin{lem}\label{compo}
We have the following substitution relation between $C(z)$ and $B(z)$:
\begin{equation}\label{CBU}
C(z) = B(U(z))\cdot (1+C(z))~.
\end{equation}
\end{lem}
To see that this identity gives a composition schema, note that it may equivalently be written as $C(z) = \hat{B}(U(z))$ with $\hat{B}=B/(1-B)$. The multiplication $(1+C(z))$ accounts for the extra submap attachment due to the split of the root edge (see the construction preceding \refP{decom2}).

The substitution relations of the two preceding lemmas 
yield, via well-known methodology, that $(M,C,H)$ and $(C,B,U)$ are both map schemata. More specifically, we have the following two propositions. 
%\section{Airy distribution for quadrangulations}\label{sec:schema}
%In this section, we describe composition schemata for both rooted connected and $2$-connected quadrangulations, and then establish corresponding distributional results. 

\begin{prop}\label{schema1}
The triple $(M,C,H)$ is a map schema with 
\begin{equation}\label{values1}
\alpha_0 = \frac7{15},~\alpha =\frac{40}{27} ,~ \beta=\frac{5^{2/3}\cdot  15}{28} ,~ \gamma=\frac{9}{5^{1/3}\cdot  4}  ~.
\end{equation}
\end{prop}

\begin{prop}\label{schema2}
The triple $(C,B,U)$ is a map schema with
\begin{equation}\label{values2}
\alpha_0 = \frac5 7,~\alpha= \frac{21^{1/2}\cdot 9}{5^{1/2}\cdot  50},~ \beta= \frac{7^{2/3}}{6^{1/3}\cdot 2}  ,~ \gamma=\frac {5}{42^{1/3}\cdot 2} ~.
\end{equation}
\end{prop}

%These two map schemata lead to the Airy distributional results stated in Propositions~\ref{airyprop2} and~\ref{airyprop1}. 

Note that Proposition~\ref{airyprop2} follows immediately from \refC{largestcor} and \refP{schema2}.

We will also need the following analogue of \refP{airyprop2} for the largest $2$-connected block of a general quadrangulation, which follows from 
 \refC{largestcor} and \refP{schema1}.
 
\begin{prop}\label{airyprop1}
Let $\bQ_q \in_u \cQ_q$, then for any $A>0$, uniformly over $x\in [-A,A]$,
\[
\p{\b(\bQ_q) = \left\lfloor 7q/15 + x q^{2/3}\right\rfloor} = \frac {\beta \cA\left(\beta x\right)}{q^{2/3}} \left( 1 + o(1) \right)~,
\]
where $\beta$ is given in \refP{schema1}.
\end{prop}

%\begin{proof}
%This follows from \refC{largestcor} and \refP{schema1}.
%\end{proof}

\begin{lem}\label{expansions}
$H(z),C(z),U(z)$, and $B(z)$ each has radius of convergence and asymptotic expansion around $r_H,r_C,r_U$, and $r_B$ as given in Table~\ref{table2}.
\end{lem}

\begin{center}
\begin{table}[htb]
\renewcommand{\arraystretch}{2}
  \begin{tabular}{| c| c  | c |}
    \hline
    $f$ & $r_f$ &  expansion at $r_f$ \\ \hline
    $H$ & $1/12$ &  $\frac{27}{196} - \frac{405}{1372}\left(1- 12z \right)+\frac{54}{343}\left(1- 12z \right)^{3/2}+O\left( (1- 12z )^2\right)$ \\ \hline
    $C$ & $27/196$ &   $\frac{1}{27}- \frac{28}{135}\left(1- 196z/27 \right)+\frac{392}{675}\sqrt{\frac 7{15}} \left(1-196z/27 \right)^{3/2}+O\left( (1- 196z/27 )^2\right)$ \\ \hline
    $U$ & $27/196$ & $ \frac{4}{27}- \frac{28}{135}\left(1- 196z/27 \right)+\frac{112}{675}\sqrt{\frac 7{15}}\left(1-196z/27 \right)^{3/2}+O\left( (1- 196z/27 )^2\right)$ \\ \hline
    $B$ & $4/27$ & $\frac1{28} - \frac{27}{196} \left(1- 27z/4 \right)+\frac{9\sqrt{3}}{98}\left(1-27z/4 \right)^{3/2}+O\left( (1- 27z/4 )^2\right)$ \\
    \hline
  \end{tabular}
    \caption{}\label{table2}
 \end{table}
\end{center}
\vspace*{-1.2cm}

This lemma may be established essentially automatically using standard techniques in enumerative combinatorics. We include a brief explanation of this methodology in Appendix~\ref{app:derivation}.

\section{Sizes and diameters of pendant submaps}\label{sec:complement}

In this section, we first obtain a size bound for the decorations of the largest simple block in a uniform rooted $2$-connected quadrangulation. Using this we then derive a corresponding diameter bound which leads to a ``GH convergence version'' of \refP{mainprop2}, shown in \refS{sec:diambound}. Analogous tail bounds for uniform rooted quadrangulations are stated in \refS{sec:analogous}.

\begin{proof}[{\bf Proof of \refP{pendant2}}]
Let $\lambda:\N\to\N$ be a function tending to infinity with $\lambda(r) \le \frac{r^{1/3}}{(\log r)^2}$. For $k\le r-r^{2/3}\lambda(r)$, the bound follows straightforwardly from \refP{airyprop2} and \refC{quotecor}. We hereafter assume that $r-r^{2/3}\lambda(r) < k \le r$. Note that for $r$ large enough, $r-r^{2/3}\lambda(r)>r/2$, so there must be less than one largest simple block of size $k$.

Let $\bR_r\in_u\cR_r$, and note that $\p{\sb(\bR_r) = k} = \frac{|\cR_{r,k}|}{|\cR_r|}$. By \refP{expansionquote} and \refL{expansions}, $|\cR_r| = [z^{r-2}]C(z)= \Theta(1)\cdot r^{-5/2} r_C^{-r}$ as $r\to\infty$ with $r_C = \frac{27}{196}$.

Each element of $\cR_{r,k}$ may be constructed by first choosing $\rS\in\cS_k$ and a collection $(\rM_e,e\in e(\rS))$ of rooted $2$-connected quadrangulations and with $\sum_{e\in e(\rS)}|e(\rM_e)| = 2(r-k)$; then attaching each $\rM_e$ to $e\in e(\rS)$ to obtain a map $R$ with $r$ edges; and, finally, fixing a root edge $\overline{e}$ in $R$ from among the $(4r-8)$ possible choices. This builds a map $(R,\overline{e})\in \cR_{r,k}$, and any element of $\cR_{r,k}$ may be so built. 
It follows that
\[
|\cR_{r,k}|\le  |\cS_k| \cdot \left(\sum_{(x_1,\ldots,x_{2k-4})} \prod_{i=1}^{2k-4} |\cR_{x_i}|\right) \cdot (4r-8)\, ,
\]
where the sum is over non-negative integer vectors $(x_1,\ldots,x_{2k-4})$ with $\sum_{i \le 2k-4} x_i = r-k$. 
It is easily verified that for all $s,t$, $|\cR_s||\cR_{t}| \le |\cR_{s+t-2}| \le |\cR_{s+t}|$, so in the above sum we always have $\prod_{i=1}^{2k-4} |\cR_{x_i}| \le |\cR_{r-k}|$. The number of summands is clearly less than $(2k-4)^{r-k}$, so we obtain 
\[
|\cR_{r,k}|\le |\cS_k| \cdot (2k-4)^{r-k} \cdot 
 |\cR_{r-k}| \cdot (4r-8)\, .
\]
Recalling that $|\cS_k|=[z^{k-2}]B(z)$, $|\cR_{r-k}|=[z^{r-k-2}]C(z)$, this yields
\[
|\cR_{r,k}| \le \Theta(r) \cdot (2k-4)^{r-k}\cdot k^{-5/2} \cdot r_B^{-k} \cdot (r-k)^{-5/2}\cdot  r_C^{-r+k}~,
\]
where $r_B$ and $r_C$ appear in Table~\ref{table2}. 

Altogether, for $r - r^{2/3} \lambda(r) < k \le r$, 
\[
\p{\sb(\bR_r) = k}= \frac{|\cR_{r,k}|}{|\cR_r|} \le   \Theta( r) \cdot (2k-4)^{r-k} \cdot k^{-5/2}\cdot (r-k)^{-5/2}\cdot r^{5/2}\cdot \left(\frac{r_C}{r_B}\right)^k~.
\]
For large enough $r$ we have $ r-k<\lambda(r) r^{2/3}\le \frac{r}{(\log r)^2}$, so for such $r$,
\[
(2k-4)^{r-k}= \exp\left((r-k)\cdot\log(2k-4)\right)\le \exp\left(\frac{r}{(\log r)^2}\cdot\log(2k-4)\right)
\le \exp\left(\frac{r}{\log r}\right)~.
\]
We have $\frac{r_C}{r_B}<1$, so there exists $b>0$ such that $\frac{r_C}{r_B}\le \exp\left(-b\right)$. It follows that for some positive constants $c_1$ and $c_2$,
\begin{align*}
\p{\sb(\bR_r) = k} \le&~ (4r-8)\cdot k^{-5/2}\cdot (r-k)^{-5/2} \cdot r^{5/2}\cdot \exp\left(-b\cdot k + \frac{r}{\log r}\right)\\
\le&~ c_1  \exp\left(-c_2 r(k/r - 5/7)^3\right)~. \qedhere
\end{align*}
\end{proof}

\begin{proof}[{\bf Proof of \refP{secondlargest}}]
For all positive integers $r$ and $x$ with $x\le r-s(r)+2$ write
\[
\cL_{r,x} = \left\{\rR\in\cR_{r,s(r)}: L(\rR) = x\right\}~.
\]
Fix $\veps\in(0,1/3)$ for the remainder of the proof. Letting $\bR_r \in_u \cR_{r,s(r)}$, 
\begin{equation}\label{eq1}
  \p{L(\bR_r) \ge r^{2/3+\veps}}
= |\cR_{r,s(r)}|^{-1} \sum_{x=\left\lceil r^{2/3+\veps} \right\rceil}^{r-s(r)+2} |\cL_{r,x}|
\end{equation}
 Since $s(r) =5 r/7 +O\left(r^{2/3}\right)$ as $r\to\infty$, by \refP{airyprop2}, 
\begin{equation}\label{eq2}
 |\cR_{r,s(r)}| = \Theta\left(r^{-2/3}\right) \cdot |\cR_r| = \Theta(r^{-2/3}) \cdot r^{-5/2} \cdot r_C^{-r} =  \Theta(r^{-19/6})\cdot r_C^{-r}~.
\end{equation}
Thus, it remains to bound $|\cL_{r,x}|$.

Each element of $\cL_{r,x}$ can be obtained by attaching some $R_x\in\cR_x$ to the largest simple block of some $\rR\in \cR_{r-x+2,s(r)}$ with $\sb(\rR)\le x$, then possibly re-assigning the root edge. We therefore have
\begin{equation}\label{Lrx}
|\cL_{r,x}| \le \Theta(r\cdot s(r))\cdot  |\cR_{r-x+2,s(r)}|\cdot |\cR_x| 
\end{equation}
as $r\to\infty$. Then to bound $|\cL_{r,x}|$, it suffices to bound $|\cR_{r-x+2,s(r)}|$ and $|\cR_x|$. For large enough $r$ and for all $x\in [r^{2/3+\veps}, r-s(r)+2]$, we have $5 (r-x+2)/7 + (r-x+2)^{2/3} \le s(r) \le r-x+2 $. For $x$ in this range, we may apply \refP{pendant2}: we obtain that for some $C'>0$, 
\begin{align*}
\frac{|\cR_{r-x+2,s(r)}|}{|\cR_{r-x+2}|} = O(1) \cdot \exp\left(-C' (r-x)\left(\frac{s(r)}{r-x} - 5/7\right)^3\right)~.
\end{align*}
For all possible $x$, by \refP{quoteprop} and \refL{expansions} we have $|\cR_{r-x+2}| = \Theta(1)\cdot (r-x)^{-5/2} r_C^{-r+x}$, so 
\[
|\cR_{r-x+2,s(r)}|=O(1)\cdot (r-x)^{-5/2}\cdot r_C^{-r+x}\cdot \exp\left(-C' (r-x)\left(\frac{s(r)}{r-x} - 5/7\right)^3\right)~.
\]
Then (\ref{Lrx}) gives
\[
|\cL_{r,x}|= O(r\cdot s(r))\cdot x^{-5/2}\cdot  (r-x)^{-5/2}\cdot  r_C^{-r}\cdot \exp\left(-C' (r-x)\left(\frac{s(r)}{r-x} - 5/7\right)^3\right)~.
\]
Since $\frac{s(r)}{r} \ge 5/7 - C r^{-1/3}$, we have for large $r$,
\begin{align*}
\exp\left(-C' (r-x)\left(\frac{s(r)}{r-x} - 5/7\right)^3\right)
\le\exp\left(-C' (r-x)^{-2}\left( 5 x/7 - C r^{2/3}\right)^3\right)~.
\end{align*}
For $r^{2/3+\veps} \le x \le r-s(r)$, and for large enough $r$, we thus have
\begin{align}
& |\cL_{r,x}|\notag\\
= &~ O(r\cdot s(r))\cdot x^{-5/2}\cdot  (r-x)^{-5/2} \cdot r_C^{-r} \cdot \exp\left(-C' (r-x)^{-2}\left(5 x /7- C r^{2/3}\right)^3\right)\notag\\
=&~ O(r\cdot s(r))\cdot x^{-5/2}\cdot  (r-x)^{-5/2} \cdot r_C^{-r} \cdot \exp\left(-C' \left(r-r^{2/3+\veps}\right)^{-2}\left(5 r^{2/3+\veps}/7 - C r^{2/3}\right)^3\right)\notag\\
=&~ r_C^{-r} \cdot \exp\left(-C'' \cdot r^{3\veps} \right)~,\label{L}
\end{align}
for some $C'' > 0$. 

Finally, combining (\ref{eq1}), (\ref{eq2}), (\ref{L}) and the fact that $s(r) = 5 r/7 + O\left(r^{2/3}\right)$, there exist positive constants $c_1$, $c_2$ and $c_3=c_3(\veps)$ such that
\[
\p{L(\bR_r) \ge r^{2/3+\veps}}
= |\cR_{r,s(r)}|^{-1} \sum_{x=\left\lceil r^{2/3+\veps} \right\rceil}^{r-s(r)+2} |\cL_{r,x}|
\le c_1  \exp\left(-c_2 r^{c_3}\right)~. \qedhere
\]\end{proof}

\subsection{Diameters of submaps pendant to the largest simple block}\label{sec:diambound}
\addtocontents{toc}{\SkipTocEntry}

We want to apply \citep{CS} to obtain a diameter bound, but first we need to transfer the diameter tail bound from \cite{CS} to the setting of $2$-connected quadrangulations.

\begin{lem}\label{2condiam}
Let $\bR_r\in_u\cR_r$, then there exist positive constants $x_0$, $c_1$ and $c_2$ such that for all $x>x_0$,
\[
\p{\diam\left(\bR_r\right) > x r^{1/4}}\le c_1 r^{2/3}\exp\left(-c_2 x\right)~.
\]
\end{lem}

\begin{proof}
For $q\in\N$, let $\bQ_{q}\in_u\cQ_q$. Given that $\b(\bQ_q)=r$, $\rR(\bQ_q)$ has the same distribution as $\bR_r$. So for all $q\ge r$ and $x>0$, we have
\begin{align*}
\p{\diam\left(\bR_r\right) > x r^{1/4}}
=&~\p{\diam\left(\rR(\bQ_q)\right) > x r^{1/4}~\big\vert~ \b(\bQ_q) = r}\\
\le&~\p{\diam\left(\bQ_q\right)> x r^{1/4}~\big\vert~ \b(\bQ_q)= r}\\
\le&~\frac{\p{\diam\left(\bQ_q\right)> x r^{1/4}}}{\p{ \b(\bQ_q)= r}}~.
\end{align*}
Now let $q=\lfloor 15r/7\rfloor$, then $xr^{1/4}\ge x(7/15)^{1/4} q^{1/4}$, so by \refP{citediam}, there exist positive constants $x_0,C_1,C_2$ such that for all $x>x_0$,
\[
\p{\diam\left(\bQ_q\right)> x r^{1/4}}\le C_1 \exp(-C_2 x)~.
\]
On the other hand, by \refP{airyprop1}, there exists $C_3>0$ such that for all $r\in\N$,
\[
\p{\b(\bQ_q) = r} \ge C_3 r^{-2/3}~.
\]
Altogether, we have
\[
\p{\diam\left(\bR_r\right) > x r^{1/4}}
\le \frac{ C_1  \exp\left(-C_2 x\right)}{C_3 r^{-2/3}}
\]
Then setting $c_1 = C_1/C_3$ and $c_2=C_2$ concludes the proof.
\end{proof}

\begin{proof}[{\bf Proof of \refP{diam}}]
Fix a positive integer $r$ and let $k\in\N$ with $k\le \min\{s(r),r-s(r)\}$. Let $\bR_r=(R_r,{e}_r)\in_u\cR_{r,s(r)}$, write $\bS_r = \rS(\bR_r)$, let ${e}'$ be the $\prec_{\bR_r}$-minimal oriented edge of $\bS_r$, and write $\bR_r' = (R_r,{e}')$. It follows from \refP{decom1} that $\bR_r'$ uniquely decomposes as $(\bS_r,\Theta)\in \cS \times \cR^{|e(\bS_r)|+1}$. Write $\Theta = (\Theta_0,\Theta_1,\ldots,\Theta_{|e(\rS(\bR_r))|})$; recall that $\Theta$ has two entries corresponding to the root edge.

For any $0\le i\le |e(\bS_r)|$, given that $|v(\Theta_i)|=k$, $\Theta_i$ is uniformly distributed over $\cR_k$. By \refL{2condiam} and since $k\le r$, there exist positive constants $x_0$, $c_1$ and $c_2$ such that for all $x\ge x_0$, and for all $0\le i\le |e(\bS_r)|$,
\begin{equation}\label{conddiambound}
 \p{\diam(\Theta_i) \ge xk^{1/4}~\big\vert~ |v(\Theta_i)| =k }\le  c_1 r^{2/3} \exp\left(-c_2 x \right)~.
\end{equation}

Note that $|e(\bS_r)|=2s(r)-4$ and recall that $D(\bR_r) = \max(\diam(\Theta_i): 0\le i\le 2s(r)-4)$. Fix $\veps \in (0,1/7)$. Using a union bound,
\begin{align*}
&\p{D(\bR_r) \ge r^{5/24}}\\
\le&~ \sum_{i=0}^{2s(r)-4} \left[ \sum_{k=1}^{\lfloor r^{2/3+\veps} \rfloor}  \p{\diam\left(\Theta_i\right) \ge r^{5/24} ,|v(\Theta_i )| =k } +\p{|v(\Theta_i)| \ge r^{2/3+\veps} }\right]\\
\le&~ \sum_{i=0}^{2s(r)-4} \left[ \sum_{k=1}^{\lfloor r^{2/3+\veps} \rfloor}  \p{\diam\left(\Theta_i\right) \ge r^{5/24}~ \bigg\vert ~|v(\Theta_i)| =k } +\p{|v(\Theta_i)| \ge r^{2/3+\veps} }\right]~.
\end{align*}
By (\ref{conddiambound}), for $k \le r^{2/3+\veps}$ and for each $0\le i\le 2s(r)-4$,
\begin{align*}
\p{\diam\left(\Theta_i \right) \ge r^{5/24} ~\bigg\vert~ |v(\Theta_i)| =k } \le&~ c_1 r^{2/3} \exp\left(-c_2 r^{5/24} k^{-1/4} \right)\\
\le&~ c_1 r^{2/3} \exp\left(-c_2 r^{1/24-\veps/4} \right)
\end{align*}
Finally, by \refP{secondlargest}, there exist positive constants $k_1,k_2$ and $k_3 = k_3(\veps)$ such that for each $0\le i\le 2s(r)-4$,
\[
\p{|v(\Theta_i)| \ge r^{2/3+\veps} } \le \p{L(\bR_r) \ge r^{2/3+\veps}} \le k_1 \exp\left(-k_2 r^{k_3}\right)~;
\]
combining the preceding $3$ inequalities and using that $s(r)\le r$ and that $1/24-\veps/4>1/168$ yields
\begin{align*}
\p{D(\bR_r)\ge r^{5/24}}\le&~ (2s(r)-3)\left[ r^{2/3+\veps}\cdot c_1 r^{2/3}\exp\left(-c_2 r^{1/24-\veps/4} \right) + k_1 \exp\left(k_2 r^{k_3}\right)\right]\\
=&~O\left(r^{7/3+\veps} \exp\left(-c_2 r^{1/168}\right) \right)+ O\left(r\cdot \exp\left(-k_2 r^{k_3}\right)\right)~.
\end{align*}
By choosing the constants $C_1,C_2$ and $C_3$ carefully, we may conclude the proof.
\end{proof}
 
 Given the diameter bound, we immediately have the ``GH convergence version'' of \refP{mainprop2}:

\begin{prop}\label{ghprop2}
Let $\bR_r\in_u \cR_{r,s(r)}$ and write $\bS_r = \rS(\bR_r)$, then as $r\to \infty$,
\begin{equation}\label{ghconv}
\left(\left( v(\bR_r), \left(\frac{21}{40r}\right)^{1/4}  \cdot d_{\bR_r}\right)~,~\left( v(\bS_r), \left(\frac{21 }{40r}\right)^{1/4} \cdot d_{\bS_r}\right)\right) \convdist \left((\cM,d),(\cM,d)\right)
\end{equation}
in distribution for the Gromov-Hausdorff topology.
\end{prop}

\begin{proof}
For any compact metric space $\rX=(X,d)$ and any subspace $\rY=(Y,d|_{Y\times Y})$ we have 
$
\dgh(\rX,\rY)\le  \sup_{x\in X} d(x,Y).
$
By \refP{diam},
$
\sup_{v\in v(\bR_r)} r^{-1/4} d_{\bR_r}(v,\bS_r)\convp 0,
$
and the result follows.
\end{proof}

\subsection{Analogous results for the largest $2$-connected block}\label{sec:analogous}
\addtocontents{toc}{\SkipTocEntry}

By analogy to Propositions~\ref{secondlargest} and~\ref{diam}, we have the following bounds for the submaps pendant to the largest $2$-connected block in a uniform quadrangulation.

Given $\rQ_q=(Q_q,e_q)\in\cQ_q$, write $\rR_q = \rR(\rQ_q)$, let $e'$ be the $\prec_{\rQ_q}$-minimal oriented edge of $\rR_q$, and write $\rQ_q'=(Q_q,e')$. By \refP{decom2}, $\rQ_q'$ uniquely decomposes as 
\[
\left(\rR_q, \left((L_i,b_i):0\le i\le 2|e(\rR_q)|-4\right) \right)~,
\]
where $L_i=(\rM_{i,j}:1\le j\le l_i)\in \cQ^{l_i}$ and $b_i=(b_{i,j}:1\le j\le l_i)\in  \{0,1\}^{l_i}$, and  $(l_i:0 \le i \le 2|e(\rR_q)|-4)$ are suitable non-negative integers. Write 
\begin{equation}\label{Lprime}
L'(\rQ_q) = \max\left\{|v(\rM_{i,j})|: 0\le i\le 2|e(\rR_q)|-4,1\le j\le l_i \right\} ~,
\end{equation}
\begin{equation}\label{Dprime}
D'(\rQ_q)= \max\left\{\diam(\rM_{i,j}): 0\le i\le 2|e(\rR_q)|-4,1\le j\le l_i \right\}~.
\end{equation}

\begin{prop}\label{secondlargest2}
For all $\veps\in(0,1/3)$, there exist positive constants $c_1$, $c_2$, and $c_3=c_3(\veps)$ such that, if $\bQ_q\in_u \cQ_{q,r(q)}$,
\[
\p{ L'(\bQ_q)\ge q^{2/3+\veps}}\le c_1  \exp\left(-c_2 q^{c_3}\right)~.
\]
\end{prop}

\begin{prop}\label{diam2}
Let $\bQ_q\in_u \cQ_{q,r(q)}$, then there exist positive constants $C_1$, $C_2$, and $C_3$ such that
\[
\p{ D'(\bQ_q) \ge q^{5/24}}\le C_1  \exp\left(-C_2 q^{C_3}\right)~.
\]
\end{prop}

We omit proofs for the above propositions since the arguments are very similar as those for Propositions~\ref{secondlargest} and \ref{diam}.

\section{Exchangeable decorations}\label{sec:exch_measures}
This section provides bounds on the Prokhorov distance between three sorts of measures on the vertices of a graph: the uniform measure, the degree-biased measure, and measures obtained by assigning vertices exchangeable random masses. In subsequent sections, these bounds help control the GHP distance between a map and its largest block. 

Recall from the introduction that for a map $G$, and $c > 0$, $cG$ denotes the measured metric space $(v(G),c\cdot d_G,\mu_G)$. 
Given a map $G$, write $\deg_G(v)$ for the degree of $v\in v(G)$ in $G$; the {\em degree-biased measure} on $G$ is the measure $\mu^B_G$ on $v(G)$ satisfying $\mu^B_G(S) = \sum_{v \in S} \mathrm{deg}_G(v)/(2|e(G)|)$. 
\begin{lem}\label{lem:unifmeas}
For any quadrangulation $Q$ and any $\eps > 0$, with $\mu_G$ and $\mu^B_G$ viewed as measures on $\eps Q$, we have 
\[
d_{\mathrm{P}}(\mu_Q,\mu_Q^B) \le \max\{\eps,1/|v(Q)|\}.
\]
\end{lem}
\begin{proof}
Let $n$ be the number of faces of $Q$, so that $|v(Q)|=n+2$ and $|e(Q)|=2n$. Fix $V \subsetneq v(Q)$ (the remaining case is trivial). A face $f$ of $Q$ is an {\em internal} face of $Q[V]$ if all vertices of $f$ lie in $V$; 
it is a {\em boundary face} of $Q[V]$ if some edge of $Q[V]$ is incident to $f$, but not all edges incident to $f$ belong to $Q[V]$.

Write $V^+=\{v\in v(Q): d_Q(v,V)\le 1\}$, and note that $V \subset V^+$. We claim that 
\begin{equation}\label{mass_ineq}
\mu_Q^B(V^+) \ge \mu_Q(V)-\frac{1}{|v(Q)|}. 
\end{equation}
If this is so then in $\eps Q$ we obtain 
$\mu_Q(V) \le \mu_Q^B(V^\eps)+1/|v(Q)|$; since $V$ was arbitrary,  the lemma then follows easily. 
We now prove (\ref{mass_ineq}).

First suppose $Q[V]$ is connected, and write $p=|V|$. If $p=1$ then the inequality is easily checked. 
If $p\ge 2$ then view $Q[V]$ as a quadrangulation with boundaries; let the boundaries have lengths $\ell_1,\ldots,\ell_k$ and write $\sum_{i=1}^k \ell_i=\ell$. 
We have $k \ge 1$ since $V\ne v(Q)$. 

Writing $i$ for the number of internal faces of $Q[V]$, Euler's formula straightforwardly yields $p=i+2+\ell/2-k$.
Furthermore, if $f$ is a boundary face of $Q[V]$ then all edges of $f$ lie within $Q[V^+]$. 
Now, a boundary face can be incident to at most two edges of $Q[V]$, so $Q[V]$ must have at least $\ell/2$ boundary faces. 
It follows that 
\[
\sum_{v \in V^+} \mathrm{deg}_Q(v) \ge \sum_{f \mbox{ internal to } Q[V^+]} 4 \ge 4(i+\ell/2) = 4(p+k-2) \ge 4(p-1). 
\]
Since $\sum_{v \in v(Q)} \mathrm{deg}_Q(v) = 2|e(Q)|=4n$, it follows that 
\[
\mu_Q^B(V) \ge \frac{p-1}{n} =\mu_Q(V)-\frac{n+2-2p}{n(n+2)} \ge \mu_Q(V) - \frac{1}{|v(Q)|}\, . 
\]
Finally, if $Q[V]$ is not connected, the same argument applied component-wise yields the same bound. 
\end{proof}
Since both $1/s \to 0$ and $(\frac{3}{8s})^{1/4} \to 0$ as $s \to \infty$, the following is immediate. 
\begin{cor}\label{cor:exchg}
For $\rS_s\in \cS_s$, with $\mu_{\rS_s}$ and $\mu_{\rS_s}^B$ viewed as measures on $\left(\frac{3}{8s}\right)^{1/4} \rS_s$, we have $d_{\rP}(\mu_{\rS_s},\mu_{\rS_s}^B) \to 0$ as $s\to\infty$. 
\end{cor}
In what follows, for a vector $\mathbf{x}=(x_1,\ldots,x_k) \in \R^k$ write $|\mathbf{x}|_p=(\sum_{i=1}^k x_i^p)^{1/p}$. 
Suppose that $\rG=(G,{e})$ is a rooted map. Enumerate the edges of $G$ as $e_1,\ldots,e_m$, where $m=|e(G)|$, and let $e_0$ be a second copy of the root edge $e$. (This makes sense even if $G$ is random, as long as it is possible to specify a canonical way to order the edges of $G$; for example, we may use the order $\prec_\rG$ described in the introduction.) 

For each $0 \le i \le m$, let $w_i$ be a uniformly random endpoint of $e_i$. Let $\bn=(n_0,\ldots,n_{m})$ be a vector of non-negative real numbers with $|\bn|_1>0$. Define a (random) probability measure $\nu_G^\bn$ on $v(G)$ as follows: for $V \subset v(G)$, let
\begin{equation}\label{eq:nu_const}
\nu_G^{\bn}(V) = \frac{1}{|\bn|_1}\cdot\sum_{\{i:w_i \in V\}} n_i.
\end{equation}
If one views $(w_i:0 \le i \le 2s-4)$ as attachment locations for pendant submaps, and $\bn$ as listing the masses of these submaps, then $\nu_{G}^{\bn}$ is the probability measure assigning each vertex $v$ a mass proportional to the total mass of submaps pendant to $v$. 

\begin{lem}\label{lem:one_or_two_deterministic}
Let $\rG=(G,{e})$ have $|e(G)|=m$ and let $\bn=(n_{0},\ldots,n_{m})$ be an exchangeable random vector of non-negative real numbers with $|\bn|_2$ strictly positive. Then for any $V\subset v(G)$,
\[
\p{|\nu_{G}^{\bn}(V)-\mu_{G}^B(V)| >\frac{2t}{|\bn|_1}+\frac{1}{m+1}~\bigg\vert~ |\bn|_2} \le 4\exp\pran{-\frac{2t^2}{|\bn|_2^2}}\, .
\]
\end{lem}
In the proof, we will use the following result of \citet[Proposition 20.6]{A}, which informally says that partial sums constructed by sampling without replacement may be obtained by first sampling with replacement and then taking a suitable projection. 
\begin{prop}[\cite{A}, Proposition 20.6]\label{aldous_prop}
Fix $x_1,\ldots,x_m \in \R$ and $k \in \{1,\ldots,m\}$, let $\sigma$ be a uniformly random permutation of $\{1,\ldots,m\}$, 
and let $I_1,\ldots,I_k$ be independent and uniform on $\{1,\ldots,m\}$. Then there exists a pair of random variables $(X,Y)$ such that $\E{Y|X}=X$ and  
\[
X \eqdist\sum_{j=1}^k x_{\sigma(j)},\quad Y\eqdist \sum_{j=1}^k x_{I_j}\, .
\]
\end{prop}
Aldous \cite{A} notes the following consequence of the preceding proposition, which is what we will in fact use.
\begin{cor}[\cite{Str}, Theorem 2]
Under the conditions of Proposition~\ref{aldous_prop}, for all continuous convex functions $\phi:\R \to \R$, 
\[
\E{\phi(X)} \le \E{\phi(Y)}\, .
\]
\end{cor}
\begin{proof}[{\bf Proof of Lemma~\ref{lem:one_or_two_deterministic}}]
Given $V \subset v(G)$, write $\partial_e V$ for the edge boundary of $S$, i.e., the set of edges $e' \in e(G)$ with one endpoint in $V$ and one in $V^c$. 
By definition, for $0 \le j \le m$, the vertex $w_j$ is a uniformly random endpoint of $e_j$. We have 
\begin{align}\label{eq:nu_xpr}
\nu_{G}^{\bn}(V) =&~ \frac{\sum_{\{j: e_j \in G[V]\}}n_{j} + \sum_{\{j: e_j \in \partial_e V\}} \I{w_j \in V} n_{j}}{|\bn|_1} ~. 
\end{align}
We now show that $\nu_{G}^{\bn}(V)$ is concentrated using Proposition~\ref{aldous_prop}. 
Independently for each $j\ge 1$ let $I_j \in_u\{0,\ldots,m\}$. 
By the exchangeability of $\bn$, it follows that for any continuous convex $\phi:\R \to \R$, 
\[
\E{\phi\bigg(\sum_{\{j: e_j \in G[V]\}}n_{j}\bigg)} \le \E{\phi\bigg(\sum_{\{j: e_j \in G[V]\}}n_{I_j}\bigg)}\, .
\]
Also by exchangeability, 
\[
\E{\sum_{\{j: e_j \in G[V]\}}n_{I_j}} = |\bn|_1 \cdot \frac{|e(G[V])|}{m+1}.
\]
Taking $\phi(x) = e^{cx}$ for $c\coloneq\frac{4t}{|\bn|_2^2}$ and applying Markov's inequality as in \citep[Theorem 2.5]{M} yields Hoeffding's inequality-type bounds for $\sum_{\{j: e_j \in G[V]\}}n_{j}$: 
\begin{align*}
& \p{\Big\vert\sum_{\{j: e_j \in G[V]\}}\frac{n_{j}}{|\bn|_1}- \frac{|e(G[V])|}{m+1}\Big\vert> \frac{t}{|\bn|_1}~\bigg\vert~|\bn|_2} \\
=&~ \p{\Big\vert\sum_{\{j: e_j \in G[V]\}} n_{j} - |\bn|_1\cdot \frac{|e(G[V])|}{m+1}\Big\vert> t~\bigg\vert~|\bn|_2} \\
\le&~ e^{-ct}\cdot \E{\exp\Big(c\cdot \Big\vert\sum_{\{j: e_j \in G[V]\}} n_{j} - |\bn|_1\cdot \frac{|e(G[V])|}{m+1}\Big\vert\Big) ~\bigg\vert~|\bn|_2}\\
\le&~ e^{-ct}\cdot \E{\exp\Big(c\cdot \Big\vert\sum_{\{j: e_j \in G[V]\}} n_{I_j} - |\bn|_1\cdot \frac{|e(G[V])|}{m+1}\Big\vert\Big) ~\bigg\vert~|\bn|_2}\\
\le  &~2\exp\pran{-\frac{2t^2}{|\bn|_2^2}}~. 
\end{align*}
The last inequality follows from a straightforward calculation; see \citep[Lemma 2.6]{M}. 
 The random variables $\I{w_j \in V}$ are iid Bernoulli$(1/2)$, so a reprise of the argument yields 
\begin{align*}
\p{\Big|\sum_{\{j: e_j \in \partial_e V\}} \frac{\I{w_j \in V} n_{j}}{|\bn|_1} - \frac{|\partial_eV|}{2(m+1)}\Big|> \frac{t}{|\bn|_1}~\bigg\vert~|\bn|_2}
& \le  2\exp\pran{-\frac{2t^2}{|\bn|_2^2}}\, .
\end{align*}
We have
\[
\mu_G^B(V) = \frac{1}{2|e(G)|}\sum_{v \in V} \mathrm{deg}(v) = \frac{2 |e(G[V])|+|\partial_e V|}{2m}, 
\]
so 
\[
\left| \mu_{G}^B(V) - \frac{|e(G[V])|}{m+1} - \frac{|\partial_e V|}{2(m+1)}\right| \le \frac{1}{m+1}\, .
\]
Considering (\ref{eq:nu_xpr}), we then have
\begin{align*}
&\p{|\nu_{G}^{\bn}(V)-\mu_{G}^B(V)| > \frac{2t}{|\bn|_1}+\frac{1}{m+1}~\bigg\vert~|\bn|_2}\\
\le&~ \p{\Big|\sum_{\{j: e_j \in \partial_e V\}} \frac{\I{w_j \in V} n_{j}}{|\bn|_1} - \frac{|\partial_eV|}{2(m+1)}\Big|> \frac{t}{|\bn|_1}~\bigg\vert~|\bn|_2} \\
\instep& + \p{\Big|\sum_{\{j: e_j \in G[V]\}}\frac{n_{j}}{|\bn|_1}- \frac{|e(G[V])|}{m+1}\Big|> \frac{t}{|\bn|_1}~\bigg\vert~|\bn|_2}\, .
\end{align*}
Combining the three probability inequalities then proves the lemma.
\end{proof}
It is easily seen that the above lemma applies even for random sets $V$, so long as $V$ is independent of the randomness used to select the endpoints $w_i$ of the edges; we will use this in what follows.

\section{Projection of masses in random quadrangulations}\label{sec:mass_proj}
In this section we apply Lemma~\ref{lem:one_or_two_deterministic} to study projection of masses in large random quadrangulations, and in particular to prove Proposition~\ref{mainprop2}. We begin by stating 
a straightforward corollary of Lemma~\ref{lem:one_or_two_deterministic}. For a metric space $\rX=(X,d)$ and $x \in X$ write $B(x,r;\rX)=\{y: d(x,y) < r\}$. 
\begin{cor}\label{cor:one_or_two}
For $s \in \N$ let $\bn_s=(n_{s,0},\ldots,n_{s,2s-4})$ be an exchangeable random vector of non-negative real numbers. 
Let $\bS_s \in_u \cS_s$, and for $v \in v(\bS_s)$ write $B(v,r) =B(v,r\cdot s^{1/4};\bS_s)$. Conditional on $\bS_s$, let $U$ and $U'$ be independent, uniformly random elements of $v(\bS_s)$. If $|\bn_s|_1 \to \infty$ and $|\bn_s|_2/|\bn_s|_1 \to 0$ then for all $x\ge0$,
\begin{equation}\label{eq:ball1}
\left\vert \nu_{\bS_s}^{\bn_s}\left(B(U,x)\right) - \mu_{\bS_s}^B\left(B(U,x)\right) \right\vert \to 0~,
\end{equation}
\begin{equation}\label{eq:ball2}
\left\vert \nu_{\bS_s}^{\bn_s}\left(B(U,x)\cap B(U',x)\right) - \mu_{\bS_s}^B\left(B(U,x)\cap B(U',x)\right) \right\vert \to 0
\end{equation}
in probability as $s\to\infty$.
\end{cor}
\begin{proof}
Fix $x \ge 0$.
We assumed that $|\bn_s|_1 \to \infty$ and $|\bn_s|_2/|\bn_s|_1 \to 0$; we may therefore choose a sequence $t(s)$ such that 
$t(s)/|\bn_s|_1 \to 0$ and $t(s)/|\bn_s|_2 \to \infty$. Now take $V=B(U,x)$. Recalling that $|e(\bS_s)|+1=2s-3$, for any $\eps > 0$, Lemma~\ref{lem:one_or_two_deterministic} gives 
\begin{align*}
& \limsup_{s \to \infty} \p{|\nu_{\bS_s}^{\bn_s}(B(U,x))-\mu_{\bS_s}^B(U,x)| \ge \eps ~\bigg\vert~|\bn_s|_2 } \\
 \le&~ \limsup_{s \to \infty}\p{|\nu_{\bS_s}^{\bn_s}(V)-\mu_{\bS_s}^B(V)| \ge \frac{2t(s)}{|\bn_s|_1}+\frac{1}{2s-3}~\bigg\vert~|\bn_s|_2} \\
	 \le&~ \limsup_{s \to \infty}  4\exp\pran{-\frac{2t(s)^2}{|\bn_s|_2^2}} \\
	= &0\, ,
\end{align*}
which is (\ref{eq:ball1}). To prove (\ref{eq:ball2}) take $V=B(U,x)\cap B(U',x)$ and argue similarly. 
\end{proof}
\begin{cor}\label{cor:degbias}
Under the assumptions of Corollary~\ref{cor:one_or_two}, with $\nu_{\bS_s}^{\bn_s}$ and $\mu_{\bS_s}^B$ viewed as measures on 
$\left(\frac{3}{8s}\right)^{1/4} \bS_s$, we have $d_{\rP}(\mu_{\bS_s},\nu_{\bS_s}^{\bn_s}) \to 0$ in probability as $s\to\infty$. 
\end{cor}

\begin{proof}
By Corollary~\ref{cor:exchg}, it suffices to show that $d_{\rP}(\mu_{\bS_s}^B,\nu_{\bS_s}^{\bn_s}) \convp 0$. 
To achieve this, we use Corollary~\ref{cor:one_or_two} and the compactness of the Brownian map $\bM=(\cM,d,\mu)$.
For the remainder of the proof we abuse notation by writing $\mu_s=\mu_{\bS_s}^B$ and $\nu_s=\nu_{\bS_s}^{\bn_s}$, for readability. 

Fix $\veps>0$. By \citep[Theorem 1]{ABA}, the triple $\left( v(\bS_s),\left(\frac{3}{8s}\right)^{1/4} d_{\bS_s}, \mu_{\bS_s}\right)$ converges in distribution to $\bM$ as $r\to\infty$. Since $\bM$ is almost surely compact and $\mu$ a.s.\ has support $\cM$, if $(U_i: i\in\N)$ are independent with law $\mu$ then we almost surely have
\[
K_{\infty}:=\inf \left\{k\in\N: \bigcup_{i=1}^k B(U_i,\veps;\bM) = \cM\right\} < \infty. 
\]
For $s\in\N$, let $(U_{s,i}: i\in\N)$ be independent with law $\mu_{\bS_s}$, and let 
\[
K_s = \inf\left\{ k\in\N: \bigcup_{i=1}^k B(U_{s,i},\veps;  (3/8s)^{1/4}\bS_s) = v(\bS_s) \right\}~. 
\] 
The aformentioned distributional convergence and the a.s.\ finiteness of $K_{\infty}$ together imply that there exists $K\in\N$ such that for all $s \in \N$, $\p{K_s>K} < \veps$. 

For $i \ge 1$ let 
\begin{equation}\label{covering}
B_i
=B\left(U_{s,i},\veps;\left(3/(8s)\right)^{1/4}\bS_s\right)\, .
\end{equation}
Let $A_1=B_1$, and for $i > 1$ let $A_i = B_i \setminus \bigcup_{j=1}^{i-1} B_j$. 
Then $A_1,\ldots,A_{K_s}$ is a covering of $v(\bS_s)$ by disjoint sets. 

Suppose that $d_{\rP}(\mu_s,\nu_s)>\veps$. Then there exists a set $S\subset v(\bS_s)$ such that either $\mu_s(S^\veps)<\nu_s(S)-\veps$ or $\nu_s(S^\veps)<\mu_s(S)-\veps$. Since $A_1,\ldots,A_{K_s}$ partition $v(\bS_s)$, there is $j$ such that either
\begin{align*}
\mu_s\left(S^\veps\cap A_j\right) & <\nu_s\left(S\cap A_j\right)-\veps/K_s \le \nu_s(A_j)-\veps/K_s  \mbox{ or }\\
\nu_s\left(S^\veps\cap A_j\right) & <\mu_s\left(S\cap A_j\right)-\veps/K_s \le \mu_s(A_j)-\veps/K_s~.
\end{align*}
For one of these to occur we must have $S\cap A_j\neq \emptyset$. Since $A_j$ has radius at most $\veps$, it follows that $A_j\subset S^\veps$. Thus, either
\[
\mu_s\left(A_j\right)  < \nu_s\left(A_j\right) -\veps/K_s \quad\mbox{ or }\quad
\nu_s\left(A_j\right)  < \mu_s\left(A_j\right) -\veps/K_s~.
\]
This yields the bound
\begin{align}
&\p{d_\rP(\mu_s,\nu_s)>\veps}\notag\\
\le&~\p{\left\vert \mu_s\left(A_j\right)-\nu_s\left(A_j\right)\right\vert>\veps/K_s
\mbox{ for some } 1\le j\le K_s}\notag\\
\le&~\sum_{j=1}^{K} \p{\left\vert \mu_s\left(A_j\right)-\nu_s\left(A_j\right)\right\vert>\veps/K_s,K_s \le K} +\p{K_s>K}\notag\\
\le&~\sum_{j=1}^{K} \p{\left\vert \mu_s\left(A_j\right)-\nu_s\left(A_j\right) \right\vert >\veps/K}+\veps\label{eq:b}~.
\end{align}

By the triangle inequality, for all $1\le i\le K$, 
\begin{align*}
&\left| \mu_s(A_i)-\nu_s(A_i)\right| \\
= &\left\vert \mu_s\left(B_i\setminus \cup_{j=1}^{i-1} B_j\right) - \nu_s\left(B_i\setminus \cup_{j=1}^{i-1} B_j\right) \right\vert\\
\le&~ \left\vert\mu_s(B_i)-\nu_s(B_i)\right\vert 
+ \left\vert \mu_s \left( B_i\cap \cup_{j=1}^{i-1} B_j\right) - \nu_s\left(B_i \cap \cup_{j=1}^{i-1} B_j\right) \right\vert \\
\le&~ \left\vert\mu_s(B_i)-\nu_s(B_i)\right\vert 
+ \sum_{j=1}^{i-1} \left\vert \mu_s \left( B_i\cap B_j\right) - \nu_s\left(B_i \cap B_j\right) \right\vert ~,
\end{align*}
where the last sum equals $0$ in the case $i=1$. 
Recalling the definitions of the $B_i$ from (\ref{covering}), the preceding inequality and Corollary~\ref{cor:one_or_two} imply that for each fixed $i\ge 1$, 
\[
\left\vert \mu_s\left(A_{i}\right)-\nu_s\left(A_{i}\right) \right\vert 
\to 0
\]
in probability as $s\to\infty$. Combining this with (\ref{eq:b}), we obtain 
\[
\limsup_{s\to\infty} \p{d_\rP(\mu_s,\nu_s)>\veps}\le \veps~.\]
Since $\veps>0$ was arbitrary, this completes the proof.
\end{proof}

We are almost ready to prove Proposition~\ref{mainprop2}; before doing so we 
state two easy facts, which each provide bounds on the GHP distance between a measured metric space and an induced (in some sense) subspace. 
The first fact is immediate from the definition of $\dghp$. 
\begin{fact}\label{ghpbound}
Fix a measured metric space $\rX=(X,d,\mu)$ and $Y \subset X$, and let $\mu_Y$ be a Borel measure on $(Y,d_Y)$, where $d_Y=d|_{Y \times Y}$ is the induced metric. Write $\rY=(Y,d_Y,\mu_Y)$. Then $\dghp(\rX,\rY) \le  \max\{d_{\rH}(\rX,\rY),d_{\rP}(\mu,\mu_Y)\}$. 
\end{fact}

The second fact informally says that in a compact measured metric space, projecting onto an $\veps$-net does not change the space very much (in the GHP sense). 
The proof is left to the reader. 

\begin{fact}\label{fact:proj}
Let $\rX=(X,d,\mu)$ be a compact measured metric space, and let $S\subset X$ be finite so that there exists $\veps>0$ with $X\subset S^\veps$. Let $(X_s:s\in S)$ be measurable subsets of $X$ such that $\bigcup_{s\in S}X_s =X$, that $\mu(X_s \cap X_{s'})=0$ for $s \ne s'$, and that $X_s \subset B(s,\eps;\rX)$ for all $s \in S$. Define a measure $\nu$ on $S$ by $\nu(s) = \mu(X_s)$ for any $s\in S$, and let $\rS=(S,d|_{S\times S},\nu)$. Then
\[
\dghp(\rX,\rS) \le \veps~.
\]
\end{fact}

\begin{proof}[{\bf Proof of Proposition~\ref{mainprop2}}]
For $r\in\N$, let $\bR_r\in_u \cR_{r,s(r)}$, and let $\bS_r= \rS(\bR_r)$. Write $\bR_r=(R_r,e_r)$, and let $e'$ be the $\prec_{\bR_r}$-minimal oriented edge of $\bS_r$. 
Next, apply the bijection of \refP{decom1} to the map $\bR_r'=(R_r,e')$: this decomposes $\bR_r$ into $\bS_r$ together with a sequence $(\Theta_i: 0\le i\le 2s(r)-4)$ of submaps of $\bR_r$. Let $n_{r,0} = |e(\Theta_0)|\I{|e(\Theta_i)|>1}$, and for $1\le i \le 2s(r)-4$ let $n_{r,i}=1+|e(\Theta_i)|\I{|e(\Theta_i)|>1}$. Then let $\bn_r=(n_{r,0},\ldots,n_{r,2s(r)-4})$ and 
construct the measure $\nu_{\bS_r}^{\bn_r}$ as in (\ref{eq:nu_const}): for $0 \le i \le 2s(r)-4$, $w_i$ is a random endpoint of $e_i$, and 
\[
\nu_{\bS_r}^{\bn_r}(v) = \frac{1}{2r-4}\sum_{\{i: w_i=v\}} n_{r,i}\, .
\]
(The difference of $1$ in the definition of $n_{r,0}$ accounts for the fact that when reconstructing $\bR_r$ from $\bS_r$ and the $\Theta_i$, we identify two copies of the root edge; the fact that $2r-4$ is the correct normalization follows from (\ref{eq:edge_count0}).)

We have $|\bn_r|_1=2r-4\to \infty$ as $r \to \infty$. Furthermore, if $L(\bR_r) \le r^{3/4}$ then $n_{r,i} \le 2r^{3/4}-3$ for all $i$, so 
$|\bn_r|_2/|\bn_r|_1 =O(r^{-1/8}) \to 0$. 
By Proposition~\ref{secondlargest}, we have $\p{L(\bR_r) \le r^{3/4}} \to 1$, so $|\bn_r|_2/|\bn_r|_1 \convp 0$. 

Corollary~\ref{cor:degbias} now implies that 
$d_{\rP}(\mu_{\bS_r},\nu_{\bS_r}^{\bn_r}) \to 0$ as $r \to \infty$, with the measures viewed as living on $\left(\frac{21}{40 r}\right)^{1/4} \bS_r$. 
For Borel measures $\mu,\nu$ on a compact metric space $(X,d)$, we have $\dghp((X,d,\mu),(X,d,\nu)) = d_{\rP}(\mu,\nu)$, 
so 
\begin{equation}\label{eq:sr_conv}
\dghp\left(\mbox{$(\frac{21}{40r})^{1/4}$} \bS_r,\pran{v(\bS_r),\mbox{$(\frac{21}{40r})^{1/4}$}\cdot d_{\bS_r},\nu_{\bS_r}^{\bn_r}}\right) \convp 0\, .
\end{equation}

We now bound the distance from $(\frac{21}{40r})^{1/4}\bR_r$ to the latter space. It is convenient to work with a graph with edge lengths rather than a finite measured metric space. 
More precisely, view each edge $e$ of $\bR_r$ as an isometric copy $I_e$ of the unit interval $[0,1]$, endowed with the rescaled Lebesgue measure 
$(2r-4)^{-1}\cdot \mathrm{Leb}_{I_e}$, and write $\bR'=(R_r',d_r',\mu_r')$ for the resulting measured metric space. 
We then have $\mu_r'(R_r')=(2r-4)^{-1}\cdot \sum_{e \in e(R_r)} \mathrm{Leb}_{I_e}(I_e) = 1$. 

We may naturally identify $v(R_r)$ with the set 
of endpoints of edges in $\bR'$, and this is an isometric embedding in that with this identification we have $d_{R_r} = d_r'|_{v(R_r)}$. Furthermore, the degree-biased measure 
$\mu_{R_r}^B$ may be obtained by projection onto $v(R_r)$: for $v \in v(R_r)$ we have $\mu_{R_r}^B(v) = \mu_r'(B(v,1/2;\bR'))= \mathrm{deg}_{R_r}(v)/(2(2s-4))$. 
By Fact~\ref{fact:proj}, it follows that for any $\eps > 0$, 
\begin{equation}\label{eq:len_to_bias}
\dghp(\veps\bR',(v(R_r),\veps d_{R_r},\mu_{R_r}^B)) \le \veps. 
\end{equation}

The space $\bS_r=(v(S_r),d_{S_r})$ is likewise isometrically embedded within $\bR'$, and we may also obtain the measure $\nu_{\bS_r}^{\bn_r}$ by projection. 
To do so, let $E_i = e(\Theta_i)$ for $1 \le i \le 2s(r)-4$, let $E_0=E(\Theta_0)\setminus \{e'\}$, and for $v \in v(S_r)$ let 
\[
X_{v} = \bigcup_{\{i: w_i=v\}} \bigcup_{e \in E_i} I_e. 
\]
Then $\nu_{\bS_r}^{\bn_r}(v) = \mu_{r}'(X_v)$. Furthermore, $(X_v:v \in v(S_r))$ covers $R_r'$ and  $\mu_r'(X_u \cap X_v)=0$ for $u \ne v$ since edges only intersect at their endpoints. 
Recalling the definition of $D(\bR_r)$ from Section~\ref{sec:sketch}, for any $v \in V$ we have $X_v \subset B(v,D(\bR_r);\bR')$. 
It follows from Fact~\ref{fact:proj} that for all $\veps > 0$, 
\begin{equation}\label{eq:len_to_pend}
\dghp(\veps\bR',(v(S_r),\veps d_{S_r},\nu_{\bS_r}^{\bn_r})) \le \veps \cdot D(\bR_r)\, .
\end{equation}
We always have $D(\bR_r) \ge 1$, so combining (\ref{eq:len_to_bias}), (\ref{eq:len_to_pend}) gives 
\[
\dghp((v(S_r),\veps d_{S_r},\nu_{\bS_r}^{\bn_r}),(v(R_r),\veps d_{R_r},\mu_{R_r}^B)) \le 2\veps \cdot D(\bR_r)
\]
Using Lemma~\ref{lem:unifmeas} to bound  $\dghp((v(R_r),\veps d_{R_r},\mu_{R_r}^B),\veps \bR_r)$, the triangle inequality then gives 
\begin{equation}\label{eq:ghpbound}
\dghp((v(S_r),\veps d_{S_r},\nu_{\bS_r}^{\bn_r}),\veps \bR_r) \le 2\veps \cdot D(\bR_r)+\max\left\{\veps,1/r\right\}~ .
\end{equation}
By Proposition~\ref{diam}, $r^{-1/4}D(\bR_r) \convp 0$, and (\ref{eq:sr_conv}) then implies that 
\[
\dghp\left(\mbox{$(\frac{21}{40 r})^{1/4}$} \bS_r, \mbox{$(\frac{21}{40 r})^{1/4}$}\bR_r\right) \convp 0\, .
\]
Since $\frac{3}{8s(r)} \cdot \bS_r \convdist \bM$ as $r \to \infty$, and $\frac{3}{8s(r)} = (1+o(1)) \frac{21}{40 r}$, the result follows. 
\end{proof}

\section{Proofs of the theorems}\label{sec:thmpf}

Recall that $\K$ is the set of measured isometry classes of compact metric spaces, and that GHP convergence refers to convergence in the the Polish space $(\K,\dghp)$. 

\begin{proof}[{\bf Proof of \refT{thm3}}]
Let $g: \K\to\R$ be a bounded continuous function, and write $\|g\| =\sup |g| < \infty$. 
Recall that $\bR_r \in_u \cR_r$, and let $\bM_r = \left(\frac{21}{40r}\right)^{1/4} \bR_r$. We show that 
\begin{equation}\label{eq:conv_bcf}
\E{g(\bM_r)} \to \E{g(\bM)}
\end{equation}
as $r \to \infty$; the result then follows by the Portmanteau theorem. 

The proof of (\ref{eq:conv_bcf}) is simply summarized: average over the size of $\sb(\bR_r)$. The details are also fairly straightforward.
Fix $\veps \in (0,1/2)$ with $\veps < 1/\|g\|$, let $\cA$ be the Airy density and let $\beta$ given by \refP{airyprop2}. Then fix $C_\veps>0$ large enough that
$\int_{-C_\veps}^{C_\veps}\beta \cA\left(\beta x \right) dx > 1-\veps$. Recall from the introduction $s(r)$ satisfies $|s(r)-5r/7| \le C r^{2/3}$ for large $r$. The constant $C$ was fixed but arbitrary, so we may assume that $C > C_{\eps}$.

Next, for $r,s\in\N$ with $s \le r$,  let $\bR_{r,s}\in_u \cR_{r,s}$ and write $\bM_{r,s} = \left(\frac{21}{40r}\right)^{1/4} \bR_{r,s}$. 
We claim that 
\begin{equation}\label{eq:conv_unif}
\sup_{\{s \in \N: |s-5r/7| \le C_{\veps} r^{2/3}\}} |g(\bM_{r,s})-g(\bM)| \to 0\, 
\end{equation}
as $r \to \infty$. Indeed: otherwise we may find a sequence $(\hat{s}(r):r \in\N)$ such that $|\hat{s}(r)-5r/7| \le C_{\veps} r^{2/3} < C r^{2/3}$ with $\limsup_{r \to \infty} |g(\bM_{r,\hat{s}(r)})-g(\bM)| \ne 0$. 
By the Portmanteau theorem, this implies that $\bM_{r,s}$ does not converge in distribution to $\bM$, contradicting \refP{mainprop2}. 
This establishes (\ref{eq:conv_unif}). 

Now for each $r\in\N$, let 
\[
E_r = \left\{ |\sb(\bR_r) - 5r/7| \le C_\veps r^{2/3}\right\}~.
\]
Recalling the definition of $\delta_s(\cdot)$ from \refT{thm:blocksizes}, it follows from \refP{airyprop2} and a Riemann approximation that for large enough $r$,
\begin{align*}
\p{E_r}
=&~ (1+o(1))\sum_{\{s \in \N: |s-5r/7| \le C_{\veps} r^{2/3}\}} \frac{
\beta \cA\left(\beta \delta_s(q)\right)}{r^{2/3}} \\
=&~ (1+o(1))\int_{-C_\veps}^{C_\veps}~ \beta\cA\left(\beta s\right) ~ds\\
 >&~ 1-2\veps. 
\end{align*}

Then for large enough $r$,
\begin{equation}\label{3bd1}
\left\vert \E{g(\bM_r)} - \E{g(\bM_r~\I{E_r})} \right\vert \le \p{E_r^c} \|g\|< 2\veps \|g\|~.
\end{equation}
We now show that $\left\vert \E{g(\bM_r~\I{E_r})} - \E{g(\bM)} \right\vert$ is also small. 
The conditional law of $\bR_r$ given that $\sb(\bR_r)=s$ is identical to that of $\bR_{r,s}$, so 
\[
\E{g(\bM_r~\I{E_r})} = \sum_{\{s \in \N: |s-5r/7| \le C_{\veps} r^{2/3}\}} \p{\sb(\bR_r)=s} \E{g(\bM_{r,s})}~.
\]
By the triangle inequality, we therefore have
\begin{align}
&\left\vert \E{g(\bM_r~\I{E_r})} - \E{g(\bM)} \right\vert \notag\\
\le&~
 \sum_{\{s \in \N: |s-5r/7| \le C_{\veps} r^{2/3}\}} \p{\sb(\bR_r)=s}\cdot  \left\vert \E{g(\bM_{r,s})} - \E{g(\bM)} \right\vert\notag\\
% &\instep + \sum_{\substack{|s-5r/7|>C_\veps r^{2/3}\\s\in\N}} \p{\sb(\bR_r)=s}\cdot \E{g(\bM)}\notag \\
 \le & \sup_{\{s \in \N: |s-5r/7| \le C_{\veps} r^{2/3}\}} |g(\bM_{r,s})-g(\bM)|\, . 
% \sum_{\substack{|s - 5r/7|\le C_\veps r^{2/3}\\s\in\N}} \p{\sb(\bR_r)=s}\cdot  \left\vert \E{g(\bM_{r,s})} - \E{g(\bM)} \right\vert\label{ta1}\\
% &\instep + \sum_{\substack{|s-5r/7|>C_\veps r^{2/3}\\s\in\N}} \p{\sb(\bR_r)=s}\cdot \E{g(\bM)}\label{ta0}~.
\end{align}
This tends to $0$ by (\ref{eq:conv_unif}), 
%so 
%\[
%\limsup_{r \to \infty} \left\vert \E{g(\bM_r~\I{E_r})} - \E{g(\bM)} \right\vert \le 2\eps\, ,
%\]
which with (\ref{3bd1}) gives $\limsup_{r \to \infty} |\E{g(\bM_r)} - \E{g(\bM)}| \le 2\veps \|g\|$. 
Since $\veps > 0$ was arbitrary, this establishes (\ref{eq:conv_bcf}) and completes the proof. 
  \end{proof}
  
  In the remaining proofs, we use the following simple fact. Recall the definition of $L'$ from (\ref{Lprime}).
\begin{fact}\label{blockfact}
Let $\rQ_q\in\cQ_q$, write $\rR_q = \rR(\rQ_q)$ and $\rS_q = \rS(\rQ_q)$. Note that if $\rS_q \neq \rS(\rR_q)$, then $\sb(\rQ_q) \ge \sb(\rR_q)$, and it follows that $\rQ_q$ contains at least two $2$-connected blocks of size at least $\sb(\rR_q)$, implying that $L'(\rQ_q)\ge \sb(\rR_q)$. 
\end{fact}

\begin{proof}[{\bf Proof of \refT{thm:blocksizes}}]
Fix $q\in\N$ and write $\bR_q = \rR(\bQ_q)$. Let $\bR\in_u\cR_{r(q)}$. Given that $\b(\bQ_q)=r(q)$, $\bR_q\eqdist\bR$, so
\begin{align*}
&\p{\b(\bQ_q)=r(q),\sb(\bR_q)=s(r(q))} \\
=&~ \p{\sb(\bR_q) = s(r(q))~\big\vert~\b(\bQ_q) = r(q)} \cdot\p{\b(\bQ_q)= r(q)}\\
=&~\p{\sb(\bR) = s(r(q))} \cdot\p{\b(\bQ_q)= r(q)}~.
\end{align*}
Writing $\beta=\frac{5^{2/3}\cdot 15}{28}$ and $\beta'=\frac{7^{2/3}}{6^{1/3}\cdot 2}$, by Propositions \ref{airyprop1} and \ref{airyprop2}, we thus have
\begin{equation}\label{condeq1}
\p{\b(\bQ_q)=r(q),\sb(\bR_q)=s(r(q))} =\frac {\beta\cA\left(\beta \delta_r(q)\right)}{q^{2/3}} \frac {\beta' \cA\left(\beta' \delta_s(q)\right)}{r(q)^{2/3}} (1+o(1)) ~.
\end{equation}

Next, 
\begin{align}
&\left\vert \p{\b(\bQ_q)=r(q),\sb(\bR_q)=s(r(q))} - \p{\b(\bQ_q)=r(q),\sb(\bQ_q)=s(r(q))}\right\vert\notag\\
\le&~\p{\b(\bQ_q)=r(q),\sb(\bQ_q)\neq s(r(q)),\sb(\bR_q)=s(r(q))}\nonumber\\
&\instep +\p{\b(\bQ_q)=r(q),\sb(\bQ_q)= s(r(q)),\sb(\bR_q)\neq s(r(q))}~,\label{diff_bound}
\end{align}
If $\{\b(\bQ_q)=r(q),\sb(\bQ_q)\neq s(r(q)),\sb(\bR_q)=s(r(q))\}$ occurs then $L'(\bQ_q) \ge s(r(q))$, as explained in Fact~\ref{blockfact}. 
Similarly, if $\{\b(\bQ_q)=r(q),\sb(\bQ_q)= s(r(q)),\sb(\bR_q)\neq s(r(q))\}$ occurs then $\bQ_q$ must contain a simple block of size $s(r(q))$ that does not lie within $\bR_q$; since 
$\b(\bQ_q)=r(q) > s(r(q))$, in this case we also obtain $L'(\bQ_q) \ge s(r(q))$. It follows from \refP{secondlargest2} that there exist positive constants $c_1,c_2,c_3$ such that 
\begin{align*}
& |\p{\b(\bQ_q)=r(q),\sb(\bR_q)=s(r(q))} - \p{\b(\bQ_q)=r(q),\sb(\bQ_q)=s(r(q))}| \\
 \le &~ 2 \p{L'(\bQ_q) \ge s(r(q))} \\
 \le& ~ c_1 \exp\left(-c_2 q^{c_3}\right) \\
=& ~ o\left(q^{-2}\right)~,
\end{align*}
which combined with (\ref{condeq1}) proves the theorem. 
\end{proof}

For the proof of \refT{thm2}, we require a lemma bounding the maximum degree in a quadrangulation uniformly drawn from $\cQ_{q,r(q),s(q)}$; the lemma follows easily from the fact that degrees in uniform quadrangulations have exponential tails. 
\begin{lem}\label{lem:deg}
Let $\bQ_{q}\in_u\cQ_{q,r(q),s(q)}$. Then for all $q$ sufficiently large, 
\[
\p{\max(\deg_{\bQ_q}(w):w \in v(\bQ_q)) \ge (\ln q)^2}<~ q^{-10}~.
\]
\end{lem}
\begin{proof}
By \citep[Theorem 2 (i)]{BC} (and Tutte's angular bijection between maps and quadrangulations), 
for all $\eps > 0$ there exists $B > 0$ such that for all $q\ge 3$, if $\bQ \in_u \cQ_q$ and $u \in_u v(\bQ)$ then 
\begin{equation}\label{in:degbd}
\p{\deg_{\bQ}(v) > d}< B\pran{\frac{1}{2}+\veps}^d~.
\end{equation}
Given that $\b(\bQ)=r(q)$ and $\sb(\rR(\bQ))=s(q)$, the conditional law of $\bQ$ is uniform on $\bQ_{q,r(q),s(q)}$; so 
\begin{align*}
	& \p{\max(\deg_{\bQ_q}(w):w \in v(\bQ_q)) > d} 	\\
 = 	&~ \Cprob{\max(\deg_{\bQ}(w):w \in v(\bQ)) > d}{\b(\bQ)=r(q),\sb(\rR(\bQ))=s(q)} \\
\le 	&~ q\cdot \Cprob{\deg_{\bQ}(u) > d}{\b(\bQ)=r(q),\sb(\rR(\bQ))=s(q)} \\
\le 	& ~q \cdot \frac{\p{\deg_{\bQ}(u) > d}}{\p{\b(\bQ)=r(q),\sb(\rR(\bQ))=s(q)}} \\
= 	&~ O(q^{7/3})\p{\deg_{\bQ}(u) > d}\, ,
\end{align*}
the final inequality by Theorem~\ref{thm:blocksizes} and the definition of $r(q)$ and $s(q)$. Taking $d=\ln^2 q$ and $\veps<1/2$, the result then follows from (\ref{in:degbd}). 
\end{proof}

\begin{proof}[{\bf Proof of \refT{thm2}}]
Recall that $\bQ_q \in_u \cQ_{q,r(q),s(q)}$, and $\bR_q =R(\bQ_q)$ and $\bS_q=S(\bQ_q)$. Let $\bQ\in_u\cQ_q$.
Given that $\b(\bQ) = r(q)$ and $\sb(\bQ) = s(q)$, we have $\bQ_q \eqdist \bQ$. By \refFt{blockfact}, we then have 
\begin{align*}
\p{\bS_q \neq \rS(\bR_q)}
=&~\p{\rS(\bQ) \neq \rS(\rR(\bQ))\big\vert \b(\bQ)=r(q), \sb(\bQ)=s(q)}\\
\le&~\p{L'(\bQ)\ge s(q) \big\vert \b(\bQ)=r(q),\sb(\bQ)=s(q)}\\
\le&~ \frac{\p{L'(\bQ)\ge s(q)\big\vert \b(\bQ)=r(q)}}{\p{\b(\bQ)=r(q),\sb(\bQ)=s(q)}}~.
\end{align*}
Combined with \refT{thm:blocksizes}, this gives 
\[
\p{\bS_q\neq \rS(\bR_q)}
=O\left(q^{4/3}\right)\cdot \p{L'(\bQ)\ge s(q)\big\vert \b(\bQ)=r(q)}~.
\]

Since $s(q)=q/3+O\left(q^{2/3}\right)$, by \refP{secondlargest2} there exist $c_2,c_3>0$ such that
\[
\p{L'(\bQ)\ge s(q)\big\vert\b(\bQ)=r(q)} = O\left(\exp\left(-c_2 q^{c_3}\right)\right)~.
\]
Hence,
\begin{equation}\label{eq:qr_sub}
\p{\bS_q \neq \rS(\bR_q)}=O\left(q^{4/3}\cdot \exp\left(-c_2 q^{c_3}\right)\right)~.
\end{equation}
Now let $\bR \in_u \cR_{r(q),s(q)}$. Given that $\bS_q = \rS(\bR_q)$, we have $\bR_q \in_u \cR_{r(q),s(q)}$, so (\ref{eq:qr_sub}) implies easily that for any bounded continuous function $g:\mathbb{K}^2 \to \R$
\begin{align*}
	 \left| \E{g\pran{\mbox{$\left(\frac{21}{40r(q)}\right)^{1/4}$} \bR_q,\mbox{$\left(\frac{21}{40r(q)}\right)^{1/4}$}\bS_q}} -\E{g\pran{\mbox{$\left(\frac{21}{40r(q)}\right)^{1/4}$} \bR,\mbox{$\left(\frac{21}{40r(q)}\right)^{1/4}$}\rS(\bR)}}\right| 
	& \to 0\ ,
\end{align*}
as $q \to \infty$. 
By \refP{mainprop2} and the Portmanteau theorem, it follows that as $r(q)\to\infty$,
\begin{equation}\label{convresize}
\left( \left(\frac{21}{40r(q)}\right)^{1/4} \bR_{q},\left(\frac{21}{40r(q)}\right)^{1/4}  \bS_{q}\right) \convdist (\bM,\bM)~.
\end{equation}
Moreover, by the definition of $r(q)$, there exist $C_1,C_2>0$ such that for all $q>0$, 
\[
 \frac{9}{8 q + C_1 q^{2/3}}\le \frac{21}{40 r(q)} \le  \frac{9}{8 q - C_2 q^{2/3}}~.
\]
From this and (\ref{convresize}) we obtain 
\begin{equation}\label{conv2}
\left( \left(\frac{9}{8q}\right)^{1/4} \bR_{q},\left(\frac{9}{8q}\right)^{1/4}  \bS_{q}\right) \convdist (\bM,\bM)
\end{equation}
as $q\to\infty$. To finish the proof, we show that also
\begin{equation}\label{eq:toshow}
\left( \left(\frac{9}{8q}\right)^{1/4} \bQ_{q},\left(\frac{9}{8q}\right)^{1/4}  \bR_{q}\right) \convdist (\bM,\bM). 
\end{equation}
Joint convergence of the triple to the limit $(\bM,\bM,\bM)$ is immediate from (\ref{conv2}) and (\ref{eq:toshow}), so it remains to prove (\ref{eq:toshow}). (Note that we may not simply invoke the result of Le Gall \cite{LG} and of Miermont \cite{Mi} to conclude that the $\left(\frac{9}{8q}\right)^{1/4}\bQ_q \convdist \bM$ since $\bQ_q$ is not uniformly distributed over $\cQ_q$, but over $\cQ_{q,r(q),s(q)}$.) The argument is similar to that in \refP{mainprop2}, and we focus on explaining the points where it differs.

Let $e'$ be the $\prec_{\bQ_q}$-minimal oriented edge of $\bR_q$; by definition, this is the root edge of $\bR_q$. 
Write $\bQ_q=(Q_q,e_q)$ and $\bR_q=(R_q,e')$. 
Also, let $\bQ_q'=(Q_q,e')$. The bijection $\psi$ from \refP{decom2} gives a decomposition of $\bQ_q'$ as
\[
\left(R_q, ((L_i,b_i):0\le i\le 2r(q)-4)\right)~,
\]
 where the $L_i = (\rM_{i,j}: 1\le j\le \ell_i)\in \cQ^{\ell_i}$ satisfy (recalling (\ref{eq:edge_count})) 
 \begin{equation}\label{eq:correct_sum}
|e(Q_q)| = |e(R_q)|+\sum_{i=0}^{|e(R_q)|} \sum_{j=1}^{\ell_i} (|e(\rM_{i,j})|+1+\I{|e(\rM_{i,j})|=1})\, ,
 \end{equation}
and  $b_i=(b_{i,j}:1\le j\le \ell_i)\in \{0,1\}^{\ell_i}$. 
 
List the elements of $e(R_q)$ as $(e_i:1 \le i \le |e(R_q)|)$ according to the order $\prec_{\bR_q}$; like in  Section~\ref{sec:mapandquad}, we view $e_i$ as oriented (we oriented so that the tail $e_i^-$ precedes the head $e_i^+$ according to the breadth-first order described in the introduction, but this is unimportant; all that matters is to have a fixed rule for choosing the orientation). Also, let $e_0$ be a copy of $e'$. Under the bijection $\psi$, for each $0 \le i \le |e(R_q)|$ and $1 \le j \le \ell_i$, the value $b_{i,j}$ indicates the endpoint $e_i$ at which $\rM_{i,j}$ is attached.

Recall that $\mu^B=\mu_{\bQ_q}^B$ is the degree-biased measure on $v(\bQ_q)$, We now compare $\mu^B$ with a random projection of $\mu^B$ onto $\bR_q$. First define a vector $\bn_q$ as follows. Let $n_0=0$ if $\ell_0=0$ and otherwise let $n_0 = \sum_{j=1}^{\ell_0} (|e(M_{0,j})|+1+\I{|e(M_{0,j}| \ne 1})$, and for $1\le i \le 2r(q)-4$ let 
\begin{equation}\label{eq:nvdef}
n_i = 1+\sum_{j=1}^{\ell_i} (|e(M_{i,j})|+1+\I{|e(M_{i,j})| \ne 1})
\end{equation}
Set $\bn_q=(n_i:0 \le i \le 2r(q)-4)$; it is immediate from \refP{decom2} that $\bn_q$ is exchangeable. 
Now define the measure $\nu^{\bn_q}=\nu_{\bR_q}^{\bn_q}$ as in (\ref{eq:nu_const}): more precisely, for each edge $e_i\in e(\bR_q)$ choose a uniformly random endpoint $w_i$ of $e_i$. Then $\nu^{\bn_q}$ is specified by letting 
\[
\nu^{\bn_q}(V) = \frac{1}{2q-4} \sum_{\{i: w_i\in V\}} n_i\, 
\]
for $V \subset v(\bR_q)$. (The fact that $2q-4$ is the correct normalizing constant follows from (\ref{eq:correct_sum}).)

If $\max(n_i:0 \le i \le 2r(q)-4)/(2q-4) \convp 0$ then $|\bn_q|_2/|\bn_q|_1 \convp 0$ and the same argument which led to Corollary~\ref{cor:degbias} gives $d_{\mathrm{P}}(\mu_{\bR_q},\nu^{\bn_q}) \convp 0$. Assuming this holds then just as in (\ref{eq:sr_conv}) we obtain
\begin{equation}\label{rq_conv}
\dghp\left(\mbox{$(\frac{9}{8q})^{1/4}$} \bR_q,\pran{v(\bR_q),\mbox{$(\frac{9}{8q})^{1/4}$}\cdot d_{\bR_q},\nu^{\bn_q}}\right) \convp 0\, .
\end{equation}
Recall the definition of $D'$ from (\ref{Dprime}). Reprising the argument for (\ref{eq:ghpbound}) now gives that for $\veps > 0$, 
\begin{equation}\label{eq:diam_bd}
\dghp((v(\bR_q),\veps d_{\bR_q},\nu^{\bn_q}),\veps \bQ_q) \le 2\veps \cdot (D'(\bQ_q)+1)+\max(\veps,1/q)~ .
\end{equation}
This has a very slightly different form from (\ref{eq:ghpbound}), where the bound was $2\veps D(\bR_r) + \max(\veps,1/r)$. 
The reason for the difference is that in the current setting, the submaps of $\bQ_q$ pendant to $\bR_q$ only attach to one end of an edge of $\bR_q$. When we project the mass to form $\nu^{\bn_q}$ we may choose the ``wrong end''. This source of error did not appear when projecting mass onto the largest simple block because the $2$-connected ``decorations'' of the largest simple block are attached at {\em both} endpoints of their respective edges. 

At any rate, by Proposition~\ref{diam2}, $D'(\bQ_q)/q^{1/4} \convp 0$, so (\ref{eq:diam_bd}) and (\ref{rq_conv}) together give  $\dghp((\frac{9}{8q})^{1/4} \bQ_q,(\frac{9}{8q})^{1/4} \bR_q) \convp0$. But by (\ref{conv2}) we know that the second argument converges to $\bM$, and (\ref{eq:toshow}) follows. 

It thus remains to prove that $\max\{n_i:0 \le i \le 2r(q)-4\}/(2q-4) \convp 0$. But this is easy: $\ell_i$ is the number of copies of a particular edge in $\bQ_q$, so $\max_{0 \le i \le 2q-4} \ell_i$ is at most $\max(\deg_{\bQ_q}(w):w \in v(\bQ_q))$. By (\ref{eq:nvdef}) we then have 
\[
\max\{n_i:0 \le i \le 2r(q)-4\} \le 1 + \max\{\deg_{\bQ_q}(w):w \in v(\bQ_q))\cdot (2+\max_{i,j} |e(M_{i,j})|\}. 
\]
By Lemma~\ref{lem:deg}, the largest degree is at most $\ln^2 q$ with high probability, and Proposition~\ref{secondlargest2} gives that $q^{-3/4}\cdot \max_{i,j} |e(M_{i,j})| \le q^{-3/4}\cdot (2L'(\bQ_q)-4) \convp 0$. The result follows. 
\end{proof}

\begin{proof}[{\bf Proof of \refT{thm1}}]
The theorem follows from Theorem~\ref{thm2} in exactly the same way as Theorem~\ref{thm3} followed from Proposition~\ref{mainprop2}, using Theorem~\ref{thm:blocksizes} in place of Proposition~\ref{airyprop2} for the averaging argument. 
\end{proof}

\appendix

\section{The remaining derivation for compositional schemata}\label{app:derivation}

In this section we explain how Propositions~\ref{schema1} and ~\ref{schema2} are derived. Though this consists in a rather classical application of analytic combinatorics machinery, we have included a reasonably detailed discussion, as we believe this may be useful for readers whose expertise is primarily probabilistic. 

We first establish a system of parameterization for $M(z)$, which is the key to showing that $M(z)$ is singular with exponent $3/2$ and to extracting the coefficients of $M(z),C(z)$ and $H(z)$.

\begin{lem}\label{lagrangean}
Let $\psi_M(t) = \frac{t(2 - 9 t) }{(1 - 3 t)^2}$, let $\phi_M(t)=\frac 1 {1-3t}$, and let $L_M(z)$ be defined by the implicit relation $L_M(z) = z \phi_M(L_M(z))$, then
\[
M(z) = \psi_M(L_M(z))~.
\]
\end{lem}

Algebraic functions with such parameterization are called {\em Lagrangean}. The proof is a textbook application of Tutte's so-called {\em quadratic method}. 
This parameterization is the one used by \citet{GJ}. It differs slightly from the original parameterization given by \citet{T} and used in \cite{BFSS}, but the two are related by a birational transformation. The derivation of \refL{lagrangean} is quite the same as that given in \citep[Proposition 1]{BFSS}, and we refer readers to that work for the idea of the proof. (Also, in \cite{BFSS} the parameterization is stated for the generating function of general maps but this is equivalent, using Tutte's angular bijection, to quadrangulations. See also \citet[Section 2.9]{GJ} for a detailed explanation of the quadratic method for map enumeration.) One may inspect the Taylor expansion of $\psi_M(L_M(z))$ at $z=0$ to conclude that this parametrization gives $M(z)=2 z+9 z^2+54 z^3+378 z^4+O\left(z^5\right)$.
\begin{cor}[\citet{T}]\label{Mexpand}
 \begin{equation}\label{Lsing}
 L_M(z) = \frac 16 - \frac 16 (1- 12z )^{1/2} ~,
\end{equation}
\begin{equation}\label{Msing}
M(z)=\frac{1}{3}-\frac{4}{3} (1- 12z)+\frac{8}{3} (1-12z)^{3/2}+ O\left((1-12z)^2\right)~.
\end{equation}
\end{cor}
In particular, $M(z)$ is singular with exponent $3/2$.
\begin{proof}
Using \refL{lagrangean}, Lagrange inversion yields the explicit formulas 
\[
L_M(z) = \frac 16 - \frac 16 (1- 12z )^{1/2} ~,\quad M(z) = -1 + \frac{1}{54z^2}(-(1-18z)+(1-12z)^{3/2})\, .
\]
Writing $y=1-12z$, the asymptotic expansion for $M$ follows easily by rearrangement. 
\end{proof}
Implicit functional equations can be used to derive asymptotic expansions in great generality, even when no closed form is available, and we exploit this machinery in the current paper. We now sketch how the method is applied in our setting in slightly more detail, referring the reader to \cite{BFSS} and \citep[VII.8]{FS} for a full exposition. 

Suppose we are given $y$ defined by an implicit formula $y(z)=z\phi(y(z))$, where $\phi$ is analytic, non-zero at $0$, has non-negative Taylor coefficients, and has $\lim_{x \to r_y} x\phi'(x)/\phi(x) > 1$, where $r_y \in (0,\infty]$ is the radius of convergence of $\phi$. (In our case, $\phi$ will always in fact be a {\em rational} function satisfying the preceding conditions.) Then, using Lagrange inversion, one obtains an asymptotic expansion of $y$ around its dominant singularity (see \citep[Section VI.7]{FS}). Given another function $m$ expressible as $m(z)=\psi(y(z))$ where $\psi$ is a rational function whose radius of convergence is at least as large as that of $y$, this yields an asymptotic expansion for $m$ as follows.

First, we locate the radius of convergence for $y$. By \citep[Theorem VI.6]{FS}, we can expand $y(z)$ as
\begin{equation}\label{Lsymbolic}
y(z) = \tau - l_{1/2} \left(1- z/r_{y}\right)^{1/2} +l_{1} \left(1- z/r_{y}\right)+ O\left(\left(1- z/r_{y}\right)^{3/2}\right),
\end{equation}
where the coefficients $l_{i/2}$ are to be determined for $i\in\N$, and $r_{y}$ and $\tau$ are determined by the equations
\[
\tau \phi'(\tau) - \phi(\tau) = 0,~r_{y} = \frac {\tau} {\phi(\tau)}~.
\]
To determine $l_{1/2}$ and $l_1$, let $h(t) = r_{y} - \frac {t} {\phi(t)}$. Then $h(\tau) = 0=h'(\tau)$, so expanding $h(t)$ around $\tau$ yields
\begin{align*}
& 1 - z/r_{y} \\
= &~\frac { h(t)}{r_{y}}\\
=&~ \frac {h''(\tau)} {2r_{y}} \left(t-\tau\right)^2+ \frac {h'''(\tau)} {6 r_{y}} \left(t-\tau\right)^3 + O\left( (t - \tau)^4 \right)\\
=&~ \frac {h''(\tau)} {2r_{y}} \left[- l_{1/2} \left(1-  z/r_{y}\right)^{1/2} +l_{1} \left(1- z/r_{y}\right) + O\left(\left(1- z/r_{y}\right)^{3/2}\right) \right]^2\\
&\instep + \frac {h'''(\tau)} {6 r_{y}}  \left[- l_{1/2} \left(1-  z/r_{y}\right)^{1/2} +l_{1} \left(1- z/r_{y}\right) + O\left(\left(1- z/r_{y}\right)^{3/2}\right) \right]^3  + O\left( (t - \tau)^4 \right)\\
=&~ \frac {h''(\tau)} {2r_{y}} l_{1/2}^2 \left(1-  z/r_{y}\right) - \left( 2 \frac {h''(\tau)} {2r_{y}} l_{1/2} l_{1} + \frac {h'''(\tau)} {6 r_{y}}  l_{1/2}^3 \right) \left( 1- z/r_{y}\right)^{3/2} + O\left( \left( 1- z/r_{y}\right)^{2} \right)~.
\end{align*}
By comparing the coefficients of the terms $(1- z/r_{y})$ we obtain 
 \begin{equation}\label{l1}
 l_{1/2} = \left(\frac {2r_{y}} {h''(\tau)}\right)^{1/2}= \left(\frac{2\phi(\tau)}{\phi''(\tau)}\right)^{1/2}\, ,
 \end{equation}
and by comparing the coefficients of the terms $(1- z/r_{y})^{3/2}$ we have
 \begin{equation}\label{l2}
 l_{1} = - \frac {h'''(\tau) l_{1/2}^3} {6 h''(\tau) l_{1/2} }\, .
 \end{equation}

Now we use the expansion (\ref{Lsymbolic}) to derive an expansion for $m(z)$ around its dominant singularity $r_m$. First, the equation $m(z) = \psi(y(z))$ and the assumption that $r_\phi \ge r_y$ together imply that $r_m=r_{y}$. In the current work, we always have that $\psi'(\tau)=0$ (indeed, this seems to generally be the case in compositional schemata involving maps); together with (\ref{Lsymbolic}), a Taylor expansion of $\psi$ around $\tau$ then yields
\begin{align*}
m(z) =&~ \psi(y(z)) \\
=&~ \psi\left( \tau -  l_{1/2} \left(1- z/r_{y} \right)^{1/2} + l_{1} (1-z/r_{y}) + O\left((1-z/r_{y})^{3/2}\right) \right) \\
=&~ \psi\left( \tau \right) + \frac 12  \psi''\left(\tau\right)\left[-  l_{1/2} \left(1- z/r_{y} \right)^{1/2} + l_{1} (1-z/r_{y}) + O\left((1-z/r_{y})^{3/2}\right) \right]^2\\
&\instep + \frac 1 6 \psi'''\left(\tau \right)\left[- l_{1/2} \left(1- z/r_{y}\right)^{1/2}  + l_{1} (1-z/r_{y}) + O\left((1-z/r_{y})^{3/2}\right) \right]^3\\
&\instep +O\left(\left(1- z/r_{y}\right)^2\right)\\
=&~\psi\left( \tau \right)  + \frac 12  \psi''\left(\tau\right) l_{1/2}^2  \left(1- z/r_{y} \right)
- \left(  \psi''\left(\tau\right) l_{1/2} l_{1} + \frac 1 6 \psi'''\left(\tau \right) l_{1/2}^3 \right)  \left(1- z/r_{y} \right)^{3/2} \\
& \instep + O\left( \left(1- z/r_{y} \right)^2\right)~.
\end{align*}
We remark that the vanishing term $\psi'\left( \tau \right)=0$ accounts for the shift of the singular exponent to $3 /2$.

Using the compositional relation given in \refL{H} together with the expansion of $L_M(z)$ given in \refC{Mexpand}, we obtain that $H(z)$ is also Lagrangean. Expanding $H(z)$ at its radius of convergence verifies the correctness of the first line of Table~\ref{table2}. We obtain expansions for $C(z),U(z)$, and $B(z)$, and thereby complete the proof of Lemma~\ref{expansions}, in a similar manner; all this is formalized in the following lemma. 

\begin{lem}\label{expansions1}
Table~\ref{table1} gives Lagrangean parameterizations for $M(z),H(z),C(z),U(z),B(z)$.
\end{lem}

\begin{proof}
Let $H(z)$ be defined as in \refP{schema1}, and let $\psi_M(t)$ be given in \refL{lagrangean}. Write $t= L_M(z)$, then
\begin{align*}
H(z)=&~ z \left(\frac 1 {1-2 z(1+ M(z)) }\right)^2
=~z \left(\frac 1 {1- 2z\left(1+ \psi_M(L_M(z))\right) }\right)^2
=-\frac{t (-1 + 3 t)^3}{(1 - 5 t + 8 t^2)^2}~;
\end{align*}
this proves the first assertion. Then taking $\psi_H(t)$ as given by Table~\ref{table1} yields $H(z) = \psi_H(L_M(z))$.

The remaining parameterizations of Table~\ref{table1} are established similarly, using (\ref{MCH}) for $C(z)$, and (\ref{CBU}) for $U(z)$ and $B(z)$. The radius of convergence and expansions around the radius in Table~\ref{table2} are derived using Lagrange inversion as in \refC{Mexpand}.\end{proof}

\begin{center}
\begin{table}[htb]
\renewcommand{\arraystretch}{2}
  \begin{tabular}{| c | c  | c |}
    \hline
    $f$ & $\phi_f$ & $\psi_f$ \\ \hline
    $H$ & $\frac 1 {1-3t}$ & $-\frac{t (-1 + 3 t)^3}{(1 - 5 t + 8 t^2)^2}$ \\ \hline
    $C$ &  $-\frac{(1 - 5 t + 8 t^2)^2}{(-1 + 3 t)^3}$ & $\frac{t^2 (-1+5t)}{(-1+3t)^3}$ \\ \hline
    $U$ & $-\frac{(1 - 5 t + 8 t^2)^2}{(-1 + 3 t)^3} $ & $-\frac{t (-1 +4t)^2}{(-1+3t)^3}$ \\ \hline
    $B$ & $-\frac{(-1 + 3 t)^3}{(-1 + 4 t)^2}$ & $\frac{t^2 (-1 + 5 t)}{(-1+4t)(1-5t+8t^2)}$ \\
    \hline
  \end{tabular}
  \caption{In this table we always have $L_f(z) = z \phi_f(L_f(z))$, where $f$ is one of the functions $H,C,U$, or $B$.}\label{table1}
\end{table}
\end{center}

\noindent {\bf Remark.} 
One of the fundamental facts of singularity analysis is that the radius of convergence of a generating function determines the exponential growth rate of the associated combinatorial family. Under Tutte's angular bijection (see \cite{T}), $2$-connected and simple quadrangulations respectively correspond to $2$-edge-connected and $2$-connected maps. In view of this, the values $r_{C} = 27/196$ and $r_B = 4/27$ agree with the known exponential growth rates for loopless bridgeless maps \citet[(7)]{WL} and for $2$-connected maps \citep[Table 2]{BFSS} (noting that the coefficients of the expansion for $B(z)$ are slightly different than in \cite{BFSS}, because in that work a single loop is counted as a $2$-connected map). 

\begin{proof}[{\bf Proofs of Propositions~\ref{schema1} and \ref{schema2}}]
We have verified that $M(z)$ and $H(z)$ are singular with exponent $3/2$ in Lemmas~\ref{Mexpand} and \ref{expansions} respectively. The facts that $H(r_H) = r_C$ and that $U(r_U) = r_B$ are immediate from the values and expansions given in Table~\ref{table2}. Thus, $(M,C,H)$ and $(C,B,U)$ are map schemata. The values claimed in (\ref{values1}) and (\ref{values2}) are then derived by routine arithmetic.
\end{proof}

\section*{Acknowledgements}
LAB was supported by an NSERC Discovery grant and an FQRNT New Researcher grant throughout this research, and thanks both institutions for their support. LAB additionally thanks the Simons Foundation and the Leverhulme Trust for their support during part of this research. YW was supported for part of this research by an FQRNT International Internship sponsored by Montr\'eal's Centre de Recherches Math\'ematiques, and thanks both institutions for their support. Both authors thank the University of Oxford and  the Newton Institute, where the final work on this manuscript was completed. Both authors additionally thank two anonymous referees for useful comments and suggestions.


\small 
\begin{thebibliography}{34} 
\providecommand{\natexlab}[1]{#1}
\providecommand{\url}[1]{\texttt{#1}}
\expandafter\ifx\csname urlstyle\endcsname\relax
  \providecommand{\doi}[1]{doi: #1}\else
  \providecommand{\doi}{doi: \begingroup \urlstyle{rm}\Url}\fi
 
\bibitem[Addario-Berry \& Albenque(2013)]{ABA}
Addario-Berry, L., \& Albenque, M. (2013). The scaling limit of random simple triangulations and random simple quadrangulations. \emph{arXiv preprint arXiv:1306.5227.}

\bibitem[Aldous(1985)]{A}
Aldous, D. J. (1985). Exchangeability and related topics (pp. 1-198). Springer Berlin Heidelberg.

\bibitem[Banderier, Flajolet, Schaeffer \& Soria(2001)]{BFSS}
Banderier, C., Flajolet, P., Schaeffer, G., \& Soria, M. (2001). Random maps, coalescing saddles, singularity analysis, and Airy phenomena. Random Structures \& Algorithms, 19(3-4), 194-246.

\bibitem[Bender \& Canfield(1989)]{BC}
Bender, E. A., \& Canfield, E. R. (1989). Face sizes of 3-polytopes. Journal of Combinatorial Theory, Series B, 46(1), 58-65.

\bibitem[Burago, Burago \& Ivanov(2001)]{BBI}
Burago, D., Burago, Y., \& Ivanov, S. (2001). A course in metric geometry (Vol. 33, pp. 371-374). Providence: American Mathematical Society.

\bibitem[Chassaing \& Schaeffer(2004)]{CS}
Chassaing, P., \& Schaeffer, G. (2004). Random planar lattices and integrated superBrownian excursion. Probability Theory and Related Fields, 128(2), 161-212.

\bibitem[Even(2011)]{E}
Even, S. (2011). Graph algorithms. Cambridge University Press.

\bibitem[Flajolet \& Sedgewick(2009)]{FS}
Flajolet, P., \& Sedgewick, R. (2009). \emph{Analytic combinatorics. Cambridge University Press}.

\bibitem[Gao \& Wormald(1999)]{GW}
Gao, Z., \& Wormald, N. C. (1999). The size of the largest components in random maps. SIAM Journal on Discrete Mathematics, 12(2), 217-228.

\bibitem[Goulden \& Jackson(1983)]{GJ}
Goulden, I.P. \& Jackson, D.M. (1983). Combinatorial Enumeration. {\emph John Wiley \& Sons}. 

\bibitem[Le Gall(2013)]{LG}
Le Gall, J. F. (2013). Uniqueness and universality of the Brownian map. The Annals of Probability, 41(4), 2880-2960.

\bibitem[McDiarmid(1998)]{M}
McDiarmid, C. (1998). Concentration. In Probabilistic methods for algorithmic discrete mathematics (pp. 195-248). Springer Berlin Heidelberg.

\bibitem[Miermont(2009)]{Mi09}
Miermont, G. (2009). Tessellations of random maps of arbitrary genus. In Annales Scientifiques de l'\'Ecole Normale Sup\'erieure (Vol. 42, No. 5, pp. 725-781).

\bibitem[Miermont(2013)]{Mi}
Miermont, G. (2013). The Brownian map is the scaling limit of uniform random plane quadrangulations. Acta mathematica, 210(2), 319-401.

\bibitem[Strassen(1965)]{Str}
Strassen, V. (1965). The existence of probability measures with given marginals. Ann. Math. Statist. 36, 423-439. 

\bibitem[Tutte(1963)]{T}
Tutte, W. T. (1963). A census of planar maps. Canad. J. Math, 15(2), 249-271.

\bibitem[Walsh \& Lehman(1975)]{WL}
Walsh, T. R., \& Lehman, A. B. (1975). Counting rooted maps by genus III: Nonseparable maps. Journal of Combinatorial Theory, Series B, 18(3), 222-259.

\end{thebibliography}
\normalsize
\end{document}